\documentclass[a4paper,11pt,leqno]{article}

\usepackage{amsmath,amssymb,eucal,amsthm}
\usepackage{bm}
\newtheorem{theorem}{Theorem}[section]
\newtheorem{lemma}[theorem]{Lemma}
\newtheorem{prop}[theorem]{Proposition}

\theoremstyle{definition}

\newtheorem{example}[theorem]{Example}

\newtheorem{remark}[theorem]{Remark}

\numberwithin{equation}{section}

\newcommand{\abs}[1]{\lvert#1\rvert}

\def\n{{\mathbb{N}}}

\def\laa{{\langle}}
\def\raa{{\rangle}}

\DeclareMathOperator*{\Res}{\mathrm{Res}}
\DeclareMathOperator*{\Lim}{\mathcal{L}\mathit{im}}
\allowdisplaybreaks

\begin{document}

\title{Hyperbolic-sine analogues of Eisenstein series, generalized Hurwitz numbers, and $q$-zeta functions}

\author{Yasushi Komori\footnote{Department of Mathematics, Rikkyo University, Nishi-Ikebukuro, Toshima-ku, Tokyo 171-8501, Japan. Email: komori@rikkyo.ac.jp}, 
Kohji Matsumoto\footnote{Graduate School of Mathematics, Nagoya University, 
Chikusa-ku, Nagoya 464-8602, Japan. Email: kohjimat@math.nagoya-u.ac.jp}
and Hirofumi Tsumura\footnote{Department of Mathematics and Information Sciences, Tokyo Metropolitan University, 1-1, Minami-Ohsawa, Hachioji, Tokyo 192-0397 
Japan. Email: tsumura@tmu.ac.jp }}

\date{}

\maketitle

\begin{abstract}
  We consider certain double series of Eisenstein type 
  involving hyperbolic-sine functions. 
  We define certain generalized Hurwitz numbers, in terms of which we evaluate 
  those double series.  Our main results can be regarded as a certain
  generalization of well-known results of Hurwitz, Herglotz, Katayama and so on.
  Our results also include recent formulas of the third-named author which are
  double analogues of the formulas of Cauchy, Mellin, Ramanujan, Berndt and
  so on, about certain Dirichlet series involving hyperbolic 
  functions. As an application, we give some evaluation formulas for $q$-zeta 
  functions at positive integers.
\end{abstract}

\baselineskip 16pt

\section{Introduction}\label{sec-1}

Let $\mathbb{N}$ be the set of natural numbers, 
$\mathbb{N}_0=\mathbb{N}\cup\{0\}$,
$\mathbb{Z}$ the ring
of rational integers, $\mathbb{Q}$ the field of rational numbers,
$\mathbb{R}$ the field of real numbers, and $\mathbb{C}$ the field of
complex numbers.

We begin with a fascinating result
of Hurwitz \cite{Hur99} (see also \cite{Hu}), which is
\begin{equation}
  \sum_{\substack{m,n \in \mathbb{Z} \\ (m,n)\neq(0,0)}}\frac{1}{(m+ni)^{k}} =\frac{(2\varpi)^{k}}{k!}H_{k}
\label{1-1}
\end{equation}
for $k\in\mathbb{N}$ with $k\geq 3$,
where $i=\sqrt{-1}$, $\varpi$ is called the lemniscate constant defined by 
$$\varpi=2\int_{0}^{1}\frac{dx}{\sqrt{1-x^4}}=\frac{\Gamma(1/4)^2}{2\sqrt{2\pi}} =2.622057\cdots,$$
and $\{ H_{m}\,|\,m \in \mathbb{N}\}$ are often called Hurwitz
numbers defined as coefficients of the Laurent series expansion of the
Weierstrass $\wp$-function, namely,
\begin{equation}
  \label{1-2}
  \begin{split}
    \wp(z) & =\frac{1}{z^2}+\sum_{\substack{\lambda \in \mathbb{Z}\varpi + \mathbb{Z}\varpi i \\ \lambda \neq 0}} \left( \frac{1}{(z-\lambda)^2}-\frac{1}{\lambda^2}\right) \\
    & =\frac{1}{z^2}+\sum_{m=1}^\infty \frac{2^{m}H_{m}}{m}\frac{z^{m-2}}{(m-2)!} 
  \end{split}
\end{equation}
(see also, for example, \cite{Ay,Ri}). Note that $H_m=0$ if $m
\not\equiv 0$ (mod $4$), because $\wp(-z)=\wp(z)$ and 
$\wp(iz)=-\wp(z)$, while $H_4=1/10$, $H_8=3/10$, etc.
Formula \eqref{1-1} can be written in terms of
Eisenstein series $G_{2k}(\tau)$ defined by
\begin{equation}
  G_{2k}(\tau)=\sum_{\substack{m,n \in \mathbb{Z} \\ (m,n)\neq(0,0)}}\frac{1}{(m+n\tau)^{2k}} \qquad (k \in \mathbb{N};\,k\geq 2) \label{1-3}
\end{equation}
for $\tau \in \mathbb{C}$ with $\Im \tau>0$ (see, for example, Koblitz
\cite{Ko}, Serre \cite{Se}).  In fact \eqref{1-1} gives, for example,
\begin{equation}
  G_4(i)=\frac{1}{15}\varpi^4,\ G_8(i)=\frac{1}{525}\varpi^8,\ G_{12}(i)=\frac{2}{53625}\varpi^{12},\ldots \label{1-4} 
\end{equation}
Generalizations of these results were given subsequently by Dintzl (to
$\mathbb{Q}(\sqrt{-2})$), Matter (to $\mathbb{Q}(\sqrt{-3})$),
Naryskina (to imaginary quadratic number fields with class number
$1$), Dintzl and Herglotz (to general imaginary quadratic number
fields). For the details of these results, see Lemmermeyer
\cite[Chapter 8]{Lem}.

On the other hand, let
\begin{equation}
 \label{1-6}
\mathfrak{E}(\xi;x,y;\omega_1,\omega_2)
  =\frac{e^{2\pi i x\xi/\omega_1}}{\omega_1}\frac{\theta'(0;\tau)\theta(\xi/\omega_1+x\tau-y;\tau)}{\theta(\xi/\omega_1;\tau)\theta(x\tau-y;\tau)}
\end{equation}
for $\xi \in \mathbb{C}$, $x,y \in \mathbb{R}$ with $0<x<1$ and 
$\omega_1,\omega_2\in \mathbb{C}$, $\omega_2/\omega_1=\tau \in \mathbb{C}$ with $\Im \tau>0$, where
$\theta(z)$ is the Jacobi theta function defined by 
\begin{equation}
\label{1-7}
  \theta(z)=
  \theta(z;\tau)=
  -i\sum_{n \in \mathbb{Z} }
  \exp
  \left(
  \pi i \left(n+\frac 12\right)^2 \tau 
  +
  2 \pi i \left( n + \frac 12\right) z
  + \pi i n
  \right),
\end{equation}
with
\begin{equation}
  \theta'(0)=2\pi e^{\pi i \tau/4}\prod_{n=1}^\infty(1-e^{2\pi i n\tau})^3=2\pi\eta(\tau)^3.
\end{equation}
Then Kronecker proved
\begin{equation}                                                                       
 \label{1-5}                                                                           
   \mathfrak{E}(\xi;x,y;\omega_1,\omega_2)=\lim_{N\to\infty}\sum_{n=-N}^N 
\lim_{M\to\infty}\sum_{m=-M}^M \frac{e^{-2\pi i(mx+ny)}}{\xi+m\omega_1+n\omega_2}    
\end{equation}
(see, for example, Weil \cite[Chapter 8]{Weil}). 

In \cite{Katayama}, 
motivated by the expression \eqref{1-5},
Katayama considered the Laurent expansion 
\begin{align}
 \label{1-9}
\mathfrak{E}(\xi;x,y;\omega_1,\omega_2)=
\frac{1}{\xi}+\sum_{j=0}^{\infty}
     \frac{{\cal H}_{j+1}(x,y;\omega_1,\omega_2)}{(j+1)!}\xi^j
\end{align}
(\cite[(2.8)]{Katayama}, where he wrote $S_{j+1}(\omega_1,\omega_2;x,y)$ instead of
${\cal H}_{j+1}(x,y;\omega_1,\omega_2)$) 
and proved that 
\begin{align}
 \label{1-10}
   {\cal H}_{j+1}(x,y;\omega_1,\omega_2)=(-1)^j (j+1)!\lim_{M\to\infty}\sum_{m=-M}^M \lim_{N\to\infty}\sum_{\substack{n=-N \\ (m,n)\neq(0,0)}}^N \frac{e^{-2\pi i(mx+ny)}}{(m\omega_1+n\omega_2)^{j+1}}
\end{align}
for $j\in \mathbb{N}$ (see \cite[Theorem 3]{Katayama}), which can be regarded as a 
generalization of \eqref{1-1}. Note that these results in the case $(x,y)=(0,0)$ 
have already been studied by Herglotz \cite{Her} as mentioned above. In fact, 
Herglotz noticed that ${\cal H}_{j+1}(0,0;\omega_1,\omega_2)$ is a kind of generalization
of Hurwitz numbers, because ${\cal H}_{4k}(0,0;1,i)=-(2\varpi)^{4k} H_{4k}$.

As another analogue of \eqref{1-1}, the third-named author \cite{TsBul} recently gave some formulas involving $\sinh x=(e^x-e^{-x})/2$, for example, 
\begin{align}
& \sum_{m\in\mathbb{Z}\setminus\{0\}}\sum_{n\in\mathbb{Z}}\frac{(-1)^n}{\sinh(m\pi){(m+ni)}^{3}} = \frac{\varpi^4}{15\pi} - \frac{7}{90}\pi^3 + \frac{1}{6}\pi^2, \label{1-11}\\
& \sum_{m\in\mathbb{Z}\setminus\{0\}}\sum_{n\in\mathbb{Z}}\frac{(-1)^n}{\sinh(m\pi){(m+ni)}^{5}} = -\frac{1}{90}\varpi^4 \pi + \frac{31}{2520}\pi^5 - \frac{7}{720}\pi^4.\label{1-11-2}
\end{align}
These can be regarded as double analogues of the following classical result given by Cauchy, Mellin, Ramanujan and so on (see \cite{Be2,Be3,Ca}):
\begin{align}
\label{1-14}
\sum_{m \in \mathbb{Z}\setminus \{0\}} \frac{(-1)^{m}}{\sinh(m\pi)m^{4k+3}}={(2\pi)^{4k+3}}\sum_{j=0}^{2k+2}(-1)^{j+1} \frac{B_{2j}(1/2)}{(2j)!}\frac{B_{4k+4-2j}(1/2)}{(4k+4-2j)!}
\end{align}
for $k\in \mathbb{N}_{0}$, where $\{ B_m(x)\}$ are Bernoulli polynomials defined by
\begin{equation}
\frac{te^{xt}}{e^t-1}=\sum_{m=0}^\infty B_m(x)\frac{t^m}{m!}. \label{1-13}
\end{equation}
Berndt \cite{BeR,Be1} recovered \eqref{1-14} and derived a number of relevant new 
formulas from a general theorem on a kind of generalized Eisenstein series in 
\cite{Be0}; yet another viewpoint is given in \cite{KMT-CRB}. 

Recently, the authors gave certain functional relations between double zeta-functions of Eisenstein type and the double series defined by
$$\sum_{m \in \mathbb{Z}\setminus \{0\}}\sum_{n \in \mathbb{Z}} \frac{(-1)^{n}}{\sinh(m\pi)m^{s_1}(m+ni)^{s_2}}.$$
By considering special values of these functional relations, we can obtain 
\eqref{1-11}, \eqref{1-11-2}, and so on (see \cite{KMT-DB}). 

It is the aim of the present paper to generalize the formulas \eqref{1-10},
\eqref{1-11}, and \eqref{1-11-2}.   After preparing some notations in Section 
\ref{sec-2}, we will state one of our main results (Theorem \ref{Main-theorem})
in Section \ref{sec-3}.   Theorem \ref{Main-theorem} asserts that Laurent expansion
coefficients ${\mathcal K}_{k,r}(x,y,z;\omega_1,\omega_2)$
of a certain function $\frak{K}_r(\xi;x,y,z;\omega_1,\omega_2)$ 
can be expressed as a double limit, involving the hyperbolic-sine function,
similar to the right-hand side of \eqref{1-10}.
The proof of Theorem \ref{Main-theorem} will be given in Section \ref{sec-3-2}.

Define Bernoulli polynomials of higher order $\{ B^{\laa r \raa}_m(z)\}$ by
\begin{equation}                                                             
\label{Ber-high}  
  \Bigl(\frac{te^{tx}}{e^t-1}\Bigr)^r=
\sum_{m=0}^\infty{B^{\laa r \raa}_m(x)}\frac{t^m}{m
!}\quad \ (r\in \mathbb{N};\,|t|<2\pi). 
\end{equation}
Note that $\{ B^{\laa r \raa}_m(x/r)\}$ $(r\in \mathbb{N})$ coincide with 
the original definition $B_m^{(r)}(x)$ of 
Bernoulli polynomials of higher order given by N{\"o}rlund \cite{Nor0} 
(p.185 in \cite{Nor0}; see also \cite{Nor}).

In Section \ref{sec-3-3} we will show that 
${\mathcal K}_{k,r}(x,y,z;\omega_1,\omega_2)$
can be written in terms of $\{ B^{\laa r \raa}_m(z)\}$ and Hurwitz functions
$\{{\mathcal H}_k(x,y;\omega_1,\omega_2)\}$
(defined in Section \ref{sec-2}).
The general statement is Theorem \ref{Main-theorem-2}, and combining it with
Theorem \ref{Main-theorem}, we can show a generalization of \eqref{1-11} and
\eqref{1-11-2}, involving the parameters $(\omega_1,\omega_2)$. 
Here we state the following typical special case corresponding to
$(\omega_1,\omega_2)=(1,i)$.
We use the notation $[x]$ (resp.$\{x\}$) which denotes the integer part
(resp. fractional part) of $x\in\mathbb{R}$.
The empty sum is to be understood as zero.

\begin{theorem} 
\label{T-1-1}
Let $r\in\mathbb{N}$. 
Assume that $k\geq3$ and $0\leq z\leq 1$, or that $k=2$ and $0<z<1$,
or that $k=1$, $0<z<1$ and $rz\not\in\mathbb{Z}$.
Then in the case $r\geq 2$ we have
\begin{equation}
    \label{eq:formula_0}
\begin{split}
&    \pi^r \sum_{\substack{(m,n)\in\mathbb{Z}^2\\m\neq 0}}
    \frac{(-1)^{rn}}{(\sinh(m\pi))^r}\frac{ e^{2\pi (m+in)r(z-1/2)}}{ (m+ni)^{k}}\\
   &=\sum_{l=r+1}^{k+r}
      \frac{(2\varpi)^lH_l}{l!}
      \frac{(2\pi)^{k+r-l}B^{\laa r \raa}_{k+r-l}(z)}{(k+r-l)!} \\
    &  +
      \sum_{l=4}^{r}
      \frac{(2\varpi)^lH_l}{l!}(2\pi)^{k+r-l}\Bigl(
      \frac{B^{\laa r \raa}_{k+r-l}(z)}{(k+r-l)!}
      -
      \sum_{j=0}^{r-l}{(-1)^j}\binom{k+j-1}{j}
      \frac{B^{\laa r \raa}_{r-j-l}(z)}{(r-j-l)!}
      \frac{B_{k+j}(\{rz\})}{(k+j)!}
      \Bigr)
      \\
      &\ -
      \frac{(2\pi)^{k+r-1}}{2}
      \Bigl(
      \frac{B^{\laa r \raa}_{k+r-2}(z)}{(k+r-2)!}
      -
      \sum_{j=0}^{r-2}{(-1)^j}\binom{k+j-1}{j}
      \frac{B^{\laa r \raa}_{r-j-2}(z)}{(r-j-2)!}
      \frac{B_{k+j}(\{rz\})}{(k+j)!}
      \Bigr)
      \\
      &\ -
      (2\pi)^{k+r}
      \Bigl(
      \frac{B^{\laa r \raa}_{k+r}(z)}{(k+r)!}
      -\sum_{j=0}^{r}{(-1)^j}\binom{k+j-1}{j}
      \frac{B^{\laa r \raa}_{r-j}(z)}{(r-j)!}
      \frac{B_{k+j}(\{rz\})}{(k+j)!}
      \Bigr)\in\mathbb{Q}[\pi,\varpi^4,z],
    \end{split}
\end{equation}
and moreover the right-hand side is of degree at most $(k+r)$, $[(k+r)/4]$ and $(k-1)$
with respect to $\pi$, $\varpi^4$ and $z$ respectively.
In the case $r=1$, we have
\begin{equation}
    \label{eq:formula_1}
\begin{split}
    \pi &  \sum_{\substack{(m,n)\in\mathbb{Z}^2\\m\neq 0}}
    \frac{(-1)^n}{\sinh(m\pi)}
    \frac{e^{2\pi (m+in)(z-1/2)}}{(m+ni)^{k}}\\
    =& 
    \sum_{l=4}^{k+1}
    \frac{(2\varpi)^lH_l}{l!}
    \frac{(2\pi)^{k+1-l} B_{k+1-l}(z)}{(k+1-l)!}-
    \frac{(2\pi)^{k}B_{k-1}(z)}{2(k-1)!}\\
    & 
    +
    \frac{(2\pi)^{k+1}}{k!}
    (B_1(z)B_k(z)-
    B_{k+1}(z))\in\mathbb{Q}[\pi,\varpi^4,z],
\end{split}
\end{equation}
and the right-hand side is of degree $(k+1)$, $[(k+1)/4]$ and $(k-1)$
with respect to $\pi$, $\varpi^4$ and $z$ respectively.
Note that, for $k=1$, the meaning of the summation on the left-hand sides of
\eqref{eq:formula_0}, \eqref{eq:formula_1} is to be understood as
$\Lim_{M,N}\sum_{\substack{-M\leq m\leq M\\-N\leq n\leq N\\m\neq 0}}$,
with the notation $\Lim_{M,N}$ to be defined at the beginning of Section
\ref{sec-3}.
\end{theorem}

For example, putting $(r,z)=(1,1/2)$ and $k=3,5$ in \eqref{eq:formula_1}, we can obtain \eqref{1-11} and \eqref{1-11-2}. Similarly \eqref{eq:formula_0} with $z=1/2$ gives explicit formulas in the case when $r\geq 2$, for example, 
$$\sum_{\substack{(m,n)\in\mathbb{Z}^2\\m\neq 0}}
\frac{1}{(\sinh(m\pi))^2 (m+ni)^2}=\frac{\varpi^4}{15\pi^2}-\frac{11}{45}\pi^2+\frac{2}{3}\pi,$$
(see Example \ref{Exam-fin}). We further give another example in the case $(\omega_1,\omega_2)=(1,\rho)$ with $\rho=e^{2\pi i/3}$ (see Example \ref{Exam-rho}). Also we recover our previous results given in \cite{TsAust} (see Example \ref{Exam-Aust}).

We prove the above theorem, Theorems \ref{Main-theorem} and \ref{Main-theorem-2} by 
considering generating functions of generalized Hurwitz numbers and applying residue 
calculus to them. Hence our method of the proof is totally different from that in 
\cite{TsBul,TsAust}. It is to be noted that a large amount of delicate arguments 
about convergence of double series are included in the course of the proof.

As an application, we consider $q$-analogues of zeta-functions (abbreviated as 
$q$-zeta functions). In the 1950's, Carlitz defined and studied $q$-Bernoulli 
numbers (see \cite{Carl}). Inspired by his study, Koblitz proposed $q$-analogues 
of zeta-functions which interpolate $q$-Bernoulli numbers. 
After their work, a lot of authors investigated various $q$-zeta functions. 
Recently Kaneko, Kurokawa and Wakayama \cite{KKW} defined a certain $q$-zeta 
function $\zeta_q(s)$ and investigated its properties.  In particular
they showed that $\lim_{q\to 1}\zeta_q(s)=\zeta(s)$ for all $s \in \mathbb{C}$ 
with $s\not=1$, where $\zeta(s)$ is the Riemann zeta-function. Furthermore 
Wakayama and Yamasaki \cite{WY} generalized $\zeta_q(s)$. Here we deduce some 
formulas for the values at positive integers of these functions from 
Theorem \ref{T-1-1} (see Proposition \ref{P-5-1}). For example, 
for $q=e^{-2\pi}<1$, we prove 
\begin{equation}
\zeta_q(2)=(1-q)^2\sum_{m=1}^\infty \frac{q^m}{(1-q^m)^2}=\frac{\left( 1-e^{-2\pi}\right)^2}{8}\left(\frac{1}{3}-\frac{1}{\pi}\right)  \label{q-val-1}
\end{equation}
(see Example \ref{Exam-5-2}).

A part of the results in the present paper has been announced in \cite{KMT-RIMS}. 

\bigskip

\section{Preliminaries} \label{sec-2}

In this section, we prepare some notations and lemmas which are necessary in this
paper.
Let $\omega_1,\omega_2\in\mathbb{C}\setminus\{0\}$ such that $\Im
(\omega_2/\omega_1)>0$.  Put $\tau=\omega_2/\omega_1$.  For
$M,N\in\mathbb{N}$, let $C_{M,N}$ be the boundary of the parallelogram
whose vertices consist of $\pm(M+1/2)\omega_1\pm(N+1/2)\omega_2$.

\subsection{Generating functions of certain Eisenstein series}
In this subsection, we review Katayama's work \cite{Katayama} with additional 
remarks.    Note that Katayama's work is based on
the classical works of Kronecker, Hurwitz and Herglotz (see
\cite{Her,Hu,Weil}).

For $-1<x,y<1$ with $(x,y)\neq(0,0)$,
let $\mathfrak{E}(\xi)=\mathfrak{E}(\xi;x,y;\omega_1,\omega_2)$ be as \eqref{1-6}.
It is easy to see that $\theta(z)$, defined by \eqref{1-7}, is an odd entire function.
In particular $\theta(0)=0$.   Also it satisfies
$$
\theta(z+1)=-\theta(z),\quad
\theta(z+\tau)=-\exp(-\pi i\tau-2\pi i z)\theta(z),
$$
which implies that $\theta(z)$ has zeros of order 1 at any lattice points on
$\mathbb{Z}+\tau\mathbb{Z}$.    It is known that there is no other zero of $\theta(z)$.
From these information we find that $\mathfrak{E}$
is a meromorphic function on $\mathbb{C}$ and 
has the quasi-periodicity
\begin{align}\label{2-1}
  \mathfrak{E}(\xi+\omega_1;x,y;\omega_1,\omega_2)&=
\mathfrak{E}(\xi;x,y;\omega_1,\omega_2)e^{2\pi ix},\\
\label{2-2}
  \mathfrak{E}(\xi+\omega_2;x,y;\omega_1,\omega_2)&=
\mathfrak{E}(\xi;x,y;\omega_1,\omega_2)e^{2\pi iy}.
\end{align}
It is also seen that
$\mathfrak{E}$ has simple poles only on $\omega_1\mathbb{Z}+\omega_2\mathbb{Z}$
and its residue at the origin is $1$.
Hence
$\mathfrak{E}$ is bounded on $\bigcup_{M,N\in\mathbb{N}}C_{M,N}$.

Since $\theta(0)=0$, the definition \eqref{1-6} is not valid for $(x,y)=(0,0)$.
In this case we define
\begin{equation}
  \mathfrak{E}(\xi)=\mathfrak{E}(\xi;0,0;\omega_1,\omega_2)
  =
  \dfrac{1}{\omega_1}
  \dfrac{\theta'(\xi/\omega_1;\tau)}
  {\theta(\xi/\omega_1;\tau)}
  +
  \frac{2\pi i\xi}{\omega_1\omega_2}.
\end{equation}
Then $\mathfrak{E}$
is a meromorphic function on $\mathbb{C}$ and 
has the (quasi-)periodicity
\begin{align}
\label{eq:pE00-1}
  \mathfrak{E}(\xi+\omega_1;0,0;\omega_1,\omega_2)&=
\mathfrak{E}(\xi;0,0;\omega_1,\omega_2)+\frac{2\pi i}{\omega_2},\\
\label{eq:pE00-2}
  \mathfrak{E}(\xi+\omega_2;0,0;\omega_1,\omega_2)&=
 \mathfrak{E}(\xi;0,0;\omega_1,\omega_2).
\end{align}
This $\mathfrak{E}$ also has simple poles only on 
$\omega_1\mathbb{Z}+\omega_2\mathbb{Z}$
and its residue at the origin is $1$.

For $-1<x,y<1$,
write the Laurent expansion of $\mathfrak{E}$ at $\xi=0$ by
\begin{equation}
 \label{def-H_k-omega}
  \mathfrak{E}(\xi;x,y;\omega_1,\omega_2)=\sum_{k=0}^\infty\frac{\mathcal{H}_k(x,y;\omega_1,\omega_2)}{k!}\xi^{k-1}.
\end{equation}
The first three coefficients in \eqref{def-H_k-omega} are given by
\begin{align}
  \label{eq:H0}
  &\mathcal{H}_0(x,y;\omega_1,\omega_2)=1,\\
  \label{eq:H1}
  &\mathcal{H}_1(x,y;\omega_1,\omega_2)=
  \begin{cases}
    \dfrac{1}{\omega_1}
    \Bigl(2\pi i x+\dfrac{\theta'( x\tau-y;\tau)}{\theta(x\tau-y;\tau)}\Bigr),\qquad &(x,y)\neq (0,0),\\
    0\qquad&(x,y)=(0,0),
  \end{cases}
\end{align}
and
\begin{align}
  \label{eq:H2}
  &\mathcal{H}_2(x,y;\omega_1,\omega_2)=
  \\
  \notag
  &
  \begin{cases}
    \dfrac{1}{\omega_1^2}
    \Bigl((2\pi i x)^2+4\pi i x
    \dfrac{\theta'( x\tau-y;\tau)}{\theta(x\tau-y;\tau)}
    +
    \dfrac{\theta''( x\tau-y;\tau)}{\theta(x\tau-y;\tau)}
    -
    \dfrac{\theta'''(0;\tau)}{3\theta'(0;\tau)}
    \Bigr),\quad&(x,y)\neq(0,0),
    \\
    \dfrac{1}{\omega_1^2}\Bigl(\dfrac{4\pi i\omega_1}{\omega_2}+\dfrac{2\theta'''(0;\tau)}{3\theta'(0;\tau)}\Bigr)
    ,\quad&(x,y)=(0,0).
  \end{cases}
\end{align}

Here we claim that
for $k\neq 1$, 
$\mathcal{H}_k(x,y;\omega_1,\omega_2)$ 
is continuous at $(x,y)=(0,0)$
with respect to $x$.
To show this claim, we use
\begin{equation}
  \label{eq:def_Ft}
  \widetilde{\mathfrak{E}}(\xi;x;\omega_1,\omega_2)=
  \\
  \begin{cases}
    \mathfrak{E}(\xi;x,0;\omega_1,\omega_2)-
\mathcal{H}_1(x,0;\omega_1,\omega_2)\qquad & (x\neq 0),
    \\[3truemm]
    \mathfrak{E}(\xi;0,0;\omega_1,\omega_2)
\qquad &(x=0).
  \end{cases}
\end{equation}
In \cite{Katayama}, Katayama considered each case individually. We prove the 
following continuity. 

\begin{lemma}\label{lem2-1}
Fix $\xi\notin\omega_1\mathbb{Z}+\omega_2\mathbb{Z}$. Then
  $\widetilde{\mathfrak{E}}(\xi;x;\omega_1,\omega_2)$ is continuous with respect to $x$ at the origin.
\end{lemma}
\begin{proof}
  We have
  \begin{multline}
    \lim_{x\to 0}\widetilde{\mathfrak{E}}(\xi;x;\omega_1,\omega_2)
  \\
  \begin{aligned}
    &=\lim_{x\to 0}
     \left(\mathfrak{E}(\xi;x,0;\omega_1,\omega_2)-\left(1+\frac{2\pi ix\xi}{\omega_1}
     \right)\mathcal{H}_1(x,0;\omega_1,\omega_2)+\frac{2\pi ix\xi}{\omega_1}
     \mathcal{H}_1(x,0;\omega_1,\omega_2)\right)
    \\
    &=\lim_{x\to 0}
    \frac{e^{2\pi i x\xi/\omega_1}}{\omega_1}
    \frac{\theta'(0;\tau)\theta(\xi/\omega_1+x\tau;\tau)-
      e^{-2\pi i x\xi/\omega_1}(1+2\pi i x\xi/\omega_1)
      \theta(\xi/\omega_1;\tau)\theta'(x\tau;\tau)}
    {\theta(\xi/\omega_1;\tau)\theta(x\tau;\tau)}
    \\
    &\qquad
    -\lim_{x\to 0}
    (1+2\pi i x\xi/\omega_1)\frac{2\pi i x}{\omega_1}    
    +\lim_{x\to 0}
\dfrac{2\pi i x\xi}{\omega_1}
\mathcal{H}_1(x,0;\omega_1,\omega_2)
    \\
    &=\lim_{x\to 0}
    \frac{e^{2\pi i x\xi/\omega_1}}{\omega_1\theta(\xi/\omega_1;\tau)\theta(x\tau;\tau)}
    \biggl(\theta'(0;\tau)
    (\theta(\xi/\omega_1;\tau)+
    \theta'(\xi/\omega_1;\tau)x\tau+O(x^2))
    \\
    &\qquad\qquad\qquad\qquad\qquad\qquad\qquad\qquad\qquad
    -\theta(\xi/\omega_1;\tau)
    (\theta'(0;\tau)+O(x^2))\biggr)
+\frac{2\pi i\xi}{\omega_1\omega_2}
    \\
    &=
    \frac{1}{\omega_1}
    \frac{\theta'(\xi/\omega_1;\tau)}
    {\theta(\xi/\omega_1;\tau)}+\frac{2\pi i\xi}{\omega_1\omega_2},
  \end{aligned}
\end{multline}
where we have used the fact $\theta(0;\tau)=\theta''(0;\tau)=0$
and
\begin{equation}
\label{eq:lim_xH}
  \lim_{x\to 0}x\mathcal{H}_1(x,0;\omega_1,\omega_2)
  =
  \dfrac{1}{\tau\omega_1}
  \lim_{x\to 0}x\tau\frac{\theta'( x\tau;\tau)}{\theta(x\tau;\tau)}
  =
  \frac{1}{\omega_2}.
\end{equation}
This implies the assertion.
\end{proof}
For $-1<x<1$, we define $\widetilde{\mathcal{H}}_k(x;\omega_1,\omega_2)$ by
\begin{equation}
  \widetilde{\mathfrak{E}}(\xi;x;\omega_1,\omega_2)
  =\sum_{k=0}^\infty\frac{\widetilde{\mathcal{H}}_k(x;\omega_1,\omega_2)}{k!}\xi^{k-1}.
\end{equation}
Then 
we have
\begin{equation}
  \widetilde{\mathcal{H}}_k(x;\omega_1,\omega_2)=
  \begin{cases}
    0\qquad&(k=1),\\
    \mathcal{H}_k(x,0;\omega_1,\omega_2)
    \qquad&(k\neq1).
  \end{cases}
\end{equation}
Using the Cauchy integral formula, we see that
\begin{equation}
  \lim_{x\to 0}\widetilde{\mathcal{H}}_k(x;\omega_1,\omega_2)=
  \widetilde{\mathcal{H}}_k(0;\omega_1,\omega_2).
\end{equation}
Then we see that $\mathcal{H}_k(x,y;\omega_1,\omega_2)$ is
continuous as required.
We call $\mathcal{H}_k(\omega_1,\omega_2)
=\mathcal{H}_k(0,0;\omega_1,\omega_2)$ the \textit{$k$th Hurwitz number} 
associated with $(\omega_1,\omega_2)$, and $\mathcal{H}_k(x,y;\omega_1,\omega_2)$
the \textit{$k$th Hurwitz function} associated with $(\omega_1,\omega_2)$ (which
may be justified by \eqref{H_k-i} below).
We quote the following.

\begin{prop}[Katayama \cite{Katayama}, Theorem 3] \label{Prop-2-2} 
Let $k$ be an integer with $k\geq2$. 
Let $-1<x,y<1$. If $k=2$, then assume $(x,y)\neq(0,0)$. 
We have
  \begin{equation}
    \mathcal{H}_k(x,y;\omega_1,\omega_2)=
    \begin{cases}
      \displaystyle
      -2!\lim_{M\to\infty}\sum_{\substack{-M\leq m\leq M}} \lim_{N\to\infty}\sum_{\substack{-N\leq n\leq N\atop (m,n)\neq (0,0)}}
      \frac{e^{2\pi i(mx+ny)}}{(m\omega_1+n\omega_2)^2}\quad&(k=2),
      \\
      \displaystyle
      -k!\sum_{(m,n)\in\mathbb{Z}^2\setminus\{(0,0)\}}
      \frac{e^{2\pi i(mx+ny)}}{(m\omega_1+n\omega_2)^k}\quad&(k\geq3).
    \end{cases}
  \label{H_k-omega}
  \end{equation}

  If $k\geq3$, then the series converges absolutely and uniformly.
\end{prop}

We can see that
\begin{equation}
\label{H_k}
  \mathcal{H}_{k}(\omega_1,\omega_2)=
  -k!\sum_{(m,n)\in\mathbb{Z}^2\setminus\{(0,0)\}}
  \frac{1}{(m\omega_1+n\omega_2)^{k}}
\end{equation}
for $k\geq3$. Moreover by \eqref{eq:H2}, we have
\begin{align}
\label{H_2}
  \mathcal{H}_{2}(\omega_1,\omega_2)&=
  -2!\lim_{x\to 0}\bigg\{
  \lim_{M\to\infty}\sum_{\substack{-M\leq m\leq M}} \lim_{N\to\infty}\sum_{\substack{-N\leq n\leq N\atop (m,n)\neq (0,0)}}
      \frac{e^{2\pi imx}}{(m\omega_1+n\omega_2)^2}\bigg\}
  \\
&\notag
  =
  \frac{1}{\omega_1^2}\Bigl(\frac{4\pi i\omega_1}{\omega_2}+\frac{2\theta'''(0;\tau)}{3\theta'(0;\tau)}\Bigr).
\end{align}
\begin{lemma}\label{lem.2.3}
For $k\in\mathbb{N}$ with $k\geq3$,
we have
\begin{equation}
\label{H_k-i}
  \mathcal{H}_{k}(1,i)
  =-(2\varpi)^{k}H_{k}
\end{equation}
and 
\begin{gather}
\label{H_2-rel}
\mathcal{H}_2(1,-1/\tau)=\tau^2 \mathcal{H}_2(1,\tau)-4\pi i\tau,\\
\label{H_2-i}
  \mathcal{H}_2(1,i)=2\pi.
\end{gather}
\end{lemma}
\begin{proof}
First, comparing \eqref{1-1} and \eqref{H_k-omega}, we have \eqref{H_k-i}.
To prove the other formulas, we recall
\begin{equation}
\label{theta-1}
  (-i\tau)^{1/2}\theta(z;\tau)=-i\exp(-i\pi z^2/\tau)\theta(-z/\tau;-1/\tau)
\end{equation}
(see \cite[Chap. 7]{Weil}).
Differentiating the both sides by $z$ and putting $z=0$, we have 
\begin{equation}
\label{theta-2}
  (-i\tau)^{1/2}\theta'(0;\tau)=\frac{i}{\tau}\theta'(0;-1/\tau).
\end{equation}
In addition, differentiating twice more, we have
\begin{equation}
\label{theta-3}
  (-i\tau)^{1/2}\theta'''(0;\tau)=\frac{6\pi}{\tau^2}\theta'(0;-1/\tau)+\frac{i}{\tau^3}\theta'''(0;-1/\tau).
\end{equation}
Dividing \eqref{theta-3} by $\theta'(0;-1/\tau)$, and using \eqref{theta-2}, we have
\begin{equation}
\label{theta-5}
  \frac{i}{\tau}\frac{\theta'''(0;\tau)}{\theta'(0;\tau)}=\frac{6\pi}{\tau^2}+\frac{i}{\tau^3}\frac{\theta'''(0;-1/\tau)}{\theta'(0;-1/\tau)}.
\end{equation}
By \eqref{H_2}, we have
\begin{equation}
\label{H_2-tau}
\begin{split}
& \mathcal{H}_2(1,\tau)=\frac{4\pi i}{\tau}+\frac{2}{3}\frac{\theta'''(0;\tau)}{\theta'(0;\tau)},\\
& \mathcal{H}_2(1,-1/\tau)=-{4\pi i}{\tau}+\frac{2}{3}\frac{\theta'''(0;-1/\tau)}{\theta'(0;-1/\tau)}.
\end{split}
\end{equation}
Therefore, combining \eqref{theta-5} and \eqref{H_2-tau}, and eliminating 
$\theta'''/\theta'$, we obtain \eqref{H_2-rel} and especially \eqref{H_2-i}.
\end{proof}

\begin{remark}
It follows from Proposition \ref{Prop-2-2} that \eqref{H_2-rel} and \eqref{H_2-i} correspond to the well-known facts 
$G_2(-1/\tau)=\tau^2G_2(\tau)+2\pi i\tau$ and $G_2(i)=-\pi$, 
where $G_2(\tau)$ is the Eisenstein series of weight $2$ defined by 
$$G_2(\tau)=\lim_{M\to\infty}\sum_{\substack{-M\leq m\leq M}} \lim_{N\to\infty}\sum_{\substack{-N\leq n\leq N\atop (m,n)\neq (0,0)}}
      \frac{1}{(m+n\tau)^2},$$
(see Serre \cite[Chapter 7, \S 4.4]{Se}). 
\end{remark}

\subsection{The generating function of Lerch zeta-functions}
Fix $z\in\mathbb{C}$. 
Let
\begin{equation}
  \label{eq:def_G}
  \mathfrak{F}(\xi)=\mathfrak{F}(\xi;z;\omega_2)=\frac{2\pi i}{\omega_2}\frac{e^{2\pi i\xi z/\omega_2}}{e^{2\pi i\xi/\omega_2}-1}.
\end{equation}
Then $\mathfrak{F}$ is a meromorphic function on $\mathbb{C}$, satisfies
\begin{equation}                                                                       
  \label{eq:Gz1}                                                                       
  \mathfrak{F}(\xi;z;\omega_2)=-\mathfrak{F}(-\xi;1-z;\omega_2),                       
\end{equation}
 and has the quasi-periodicity
\begin{equation}\label{eq:perio}
  \mathfrak{F}(\xi+\omega_2)=\mathfrak{F}(\xi)e^{2\pi i z}.
\end{equation}
Furthermore $\mathfrak{F}$
has simple poles only on
$\omega_2\mathbb{Z}$ and its residue at the origin is $1$.
Write the Laurent expansion of $\mathfrak{F}$ at $\xi=0$ by
\begin{equation}\label{bernoulli}
  \mathfrak{F}(\xi;z;\omega_2)=\sum_{k=0}^\infty\frac{\mathcal{B}_k(z;\omega_2)}{k!}\xi^{k-1}.
\end{equation}
Since $\mathfrak{F}$ is essentially the generating function of Bernoulli polynomials
(see \eqref{1-13}),
we can easily obtain the following lemma (see, for example, \cite[p. 267]{Apos}).

\begin{prop}\label{prop:bernoulli}
Let $k$ be a positive integer. 
Let $z\in\mathbb{R}$. If $k=1$, then assume $z\not\in\mathbb{Z}$.
We have
  \begin{equation}\label{bernoulli-2}
    \mathcal{B}_k(\{z\};\omega_2)=\left(\frac{2\pi i}{\omega_2}\right)^k B_k(\{z\})=
    \begin{cases}
      \displaystyle
      -\lim_{N\to\infty}1!\sum_{\substack{-N\leq m\leq N\\m\neq 0}}
      \frac{e^{2\pi i mz}}{m\omega_2}\qquad&(k=1),
      \\
      \displaystyle
      -k!\sum_{m\in\mathbb{Z}\setminus\{0\}}
      \frac{e^{2\pi imz}}{(m\omega_2)^k}\qquad&(k\geq2).
    \end{cases}
  \end{equation}
If $k\geq2$, then the series converges absolutely and uniformly.
\end{prop}
Furthermore, for $r\in \mathbb{N}$, we define 
$\mathcal{B}^{\laa r \raa}_j(z;\omega_2)$ by
\begin{equation}
\label{eq:Tay_F}
  \mathfrak{F}(\xi;z;\omega_2)^r=\sum_{j=0}^\infty
\frac{\mathcal{B}^{\laa r \raa}_j(z;\omega_2)}{j!}\xi^{j-r}.
\end{equation}
Then, comparing with \eqref{Ber-high}, we have 
\begin{equation} 
\label{2-33}
\mathcal{B}^{\laa r \raa}_j(z;\omega_2)=(2\pi i/\omega_2)^jB^{\laa r \raa}_j(z).
\end{equation}

\section{A generalization of the Hurwitz-Herglotz-Katayama formula} \label{sec-3}
To state our result, we first define hyperbolic-sine analogues of 
Eisenstein series.
Let $-1<x,y<1$, $0\leq z\leq1$, $k,r\in\mathbb{N}$, 
$\omega_1,\omega_2\in \mathbb{C}$ and $\tau=\omega_2/\omega_1$ with $\Im \tau>0$.
Let $\lim_{\substack{M\to\infty\\N\to\infty}}$ stand for the limit in the following
sense: there exist sequences $\{M_k\}_{k=1}^{\infty}$, $\{N_k\}_{k=1}^{\infty}$
such that, for any $R>0$, one can find $K=K(R)$ for which $M_k\geq R$,
$N_k\geq R$ hold for any $k\geq K$.
Denote by $\Lim_{M,N}$
any one of the following limits:
\begin{equation}
\label{eq:limit}
  \lim_{\substack{M\to\infty\\N\to\infty}},\qquad
  \lim_{M\to\infty}\lim_{N\to\infty},\qquad
  \lim_{N\to\infty}\lim_{M\to\infty}.
\end{equation}

Define 
\begin{equation}
\label{Gene-Hur}
\begin{split}
  & \mathcal{G}_{k}^{\langle r\rangle}(x,y,z;\omega_1,\omega_2) \\
  &  =
    \begin{cases}
      & \displaystyle{
      \Lim_{M,N}
      \sum_{\substack{-M\leq m\leq M\\-N\leq n\leq N\\ m\neq 0}}
      \frac{(-1)^{rn}}{(\sinh(m\pi i/\tau))^r}
      \frac{e^{2\pi i(m(x+r(z-1/2)/\tau)+n(y+r(z-1/2)))}}{(m\omega_1+n\omega_2)^{k}}
      \quad}\\
      &  \qquad \qquad \Biggl(\begin{aligned}
        &k=1\text{ and }0<z<1,y+rz\not\in\mathbb{Z}, \text{ or }\\
        &k=2, (x,y)\neq(0,0)\text{ and }z=0,1
         \end{aligned}\Biggr),\\
      & \displaystyle{
      \sum_{\substack{(m,n)\in\mathbb{Z}^2\\m\neq 0}}
      \frac{(-1)^{rn}}{(\sinh(m\pi i/\tau))^r}
      \frac{e^{2\pi i(m(x+r(z-1/2)/\tau)+n(y+r(z-1/2)))}}{(m\omega_1+n\omega_2)^{k}}
      \quad }\\
      & \qquad \qquad \Biggl(\begin{aligned}
        &k\geq 3,\text{ or } \\
        &k=2 \text{ and } 0<z<1
      \end{aligned}\Biggr).
    \end{cases}
\end{split}
\end{equation}
In Theorem \ref{Main-theorem},
we will show that the first sum of
the right-hand side does not depend on a choice of the limits
\eqref{eq:limit}
while
the second converges absolutely.
It should be noted that when $(x,y,z)=(0,0,1/2)$ and $k\geq 3$ we have
\begin{equation*}
\begin{split}
  & \mathcal{G}_{k}^{\langle r\rangle}(0,0,1/2;\omega_1,\omega_2) =\frac{1}{\omega_1^k}\displaystyle{
      \sum_{\substack{(m,n)\in\mathbb{Z}^2\\m\neq 0}}
      \frac{1}{(\sinh((m+n\tau)\pi i/\tau))^r}
      \frac{1}{(m+n\tau)^{k}}
      },
\end{split}
\end{equation*}
because we see that $(\sinh((m+n\tau)\pi i/\tau))^r=(-1)^{rn}(\sinh(m\pi i/\tau))^r$. 
Hence we can call $\mathcal{G}_{k}^{\langle r\rangle}(x,y,z;\omega_1,\omega_2)$ a 
\textit{hyperbolic-sine analogue of ordinary Eisenstein series} in the sense similar 
to $q$-analogues of zeta-functions (see \eqref{5-1}-\eqref{5-3}). 

For $-1<x,y<1$ and $0\leq z\leq1$ with $(x,y,z)\neq(0,0,1)$,
define
\begin{equation}\label{K_r}
  \mathfrak{K}_r(\xi)=
  \mathfrak{K}_r(\xi;x,y,z;\omega_1,\omega_2)
=
  \Res_{\eta=0}
  \bigl(
  \mathfrak{D}_r(\xi)\eta^{-1}-
  \mathfrak{D}_r(\eta)
  \mathfrak{F}(\xi-\eta;\{y+rz\};\omega_2)
  \bigr)
\end{equation}
and 
\begin{equation}
  \mathfrak{K}_r(\xi)=
  \mathfrak{K}_r(\xi;0,0,1;\omega_1,\omega_2)
=
(-1)^{r+1}\mathfrak{K}_r(-\xi;0,0,0;\omega_1,\omega_2)
\end{equation}
where
the residue is taken
for $\xi\not\in\omega_2\mathbb{Z}$,
and
\begin{equation}
  \mathfrak{D}_r(\xi)=\mathfrak{D}_r(\xi;x,y,z;\omega_1,\omega_2)=\mathfrak{E}(\xi;x,y;\omega_1,\omega_2) \mathfrak{F}(\xi;z;\omega_2)^r.
\end{equation}
We will show later in this section
the continuity of $\mathfrak{K}_r$ with respect
to $x$ at $x=0$ when $y=0$.

The following theorem implies that $\mathfrak{K}_r$ is essentially the generating
function of $\mathcal{G}_{k}^{\langle r\rangle}$.


\begin{theorem}
  \label{Main-theorem}
For $-1<x,y<1$ and $0\leq z\leq1$,
the function $\mathfrak{K}_r(\xi;x,y,z;\omega_1,\omega_2)$ is meromorphic in $\xi$,
and especially holomorphic at $\xi=0$.   
Write the Taylor expansion of $\mathfrak{K}_r$ at $\xi=0$ by
  \begin{equation}
    \label{eq:K_taylor}
    \mathfrak{K}_r(\xi;x,y,z;\omega_1,\omega_2)=\sum_{k=1}^\infty \frac{\mathcal{K}_{k,r}(x,y,z;\omega_1,\omega_2)}{k!}\xi^{k-1}.
  \end{equation}
Let $k$ be a positive integer.
If $k=2$ and $z=0,1$, then assume $(x,y)\neq(0,0)$.
If $k=1$, then assume $0<z<1,y+rz\not\in\mathbb{Z}$.
Then we have
  \begin{equation}
  \label{K_kq-1}
    \mathcal{K}_{k,r}(x,y,z;\omega_1,\omega_2)
    =
    -k! 
    \left(\frac{\pi i}{\omega_2}\right)^r \mathcal{G}_{k}^{\langle r\rangle}(x,y,z;\omega_1,\omega_2).
  \end{equation}
The first sum of the right-hand side of \eqref{Gene-Hur}
does not depend on a choice of the limits
\eqref{eq:limit} while
the second 
converges absolutely uniformly.
\end{theorem}

\begin{remark}
Comparing \eqref{eq:K_taylor} with \eqref{def-H_k-omega}, we may call
$\mathcal{K}_{k,r}(x,y,z;\omega_1,\omega_2)$ the \textit{$k$th generalized Hurwitz
function}.
\end{remark}

The above theorem is regarded as a hyperbolic-sine analogue of the
Hurwitz-Herglotz-Katayama formula \eqref{1-10}.
In fact, though the above theorem is proved for $r\in\mathbb{N}$, let us consider
the case $r=0$ of the theorem formally.   Then
$\frak{D}_0(\xi)=\frak{E}(\xi)$, and so
$\frak{K}_0(\xi)=\frak{E}(\xi)-\frak{F}(\xi;\{y\};\omega_2)$.
Therefore from \eqref{1-9} and \eqref{bernoulli} it follows that
$$
\mathcal{K}_{k,0}(x,y,z;\omega_1,\omega_2)={\cal H}_k(x,y;\omega_1,\omega_2)
-\mathcal{B}_k(\{y\};\omega_2)
$$
and hence, using \eqref{1-10} and \eqref{bernoulli-2}, we find that
\eqref{K_kq-1} is valid (formally) for $r=0$.   (On the right-hand side of
\eqref{1-10}, the double sum is divided into two parts corresponding to
$m\neq 0$ and $m=0$, each of which is equal to
$-k!\mathcal{G}_{k}^{\langle 0\rangle}$ 
and $\mathcal{B}_k$, respectively.)
Our proof of Theorem \ref{Main-theorem} presented in Section \ref{sec-3-2} is
not valid in the case $r=0$ as it is, but it is possible to modify the argument
there slightly to obtain the proof for $r=0$.   In this sense, Theorem
\ref{Main-theorem} is a generalization of the Hurwitz-Herglotz-Katayama formula. 
\bigskip

Now we give the proof of the continuity of $\mathfrak{K}_r$.
\begin{lemma}
  \label{lm:Laurent}
  Let $k$ be an integer with $0\leq k\leq r-1$ and $X,Y\in\mathbb{C}$, $e^X\neq 1$. Then
  \begin{equation}
    \Res_{Y=0}\frac{e^{k Y}}{(e^Y-1)^r(e^{X-Y}-1)}=\frac{e^{k X}}{(e^X-1)^r}.
  \end{equation}
\end{lemma}
\begin{proof}
Choose a sufficiently large $R>|X|$, and consider the integral
  \begin{equation}
    \frac{1}{2\pi i}\int_{\cal L}
    \frac{e^{k Y}dY}{(e^Y-1)^r(e^{X-Y}-1)},
  \end{equation}
where the path of integration ${\cal L}$ is the
rectangle whose vertices are $\pm R+ci$ and $\pm R+(2\pi+c)i$, where $c\in\mathbb{R}$
is chosen so that it is not congruent to 0, $\Im X$ mod $2\pi\mathbb{Z}$.   Then
the poles inside ${\cal L}$ are 0 and $X$ (mod $2\pi i\mathbb{Z}$), and so
  \begin{multline}
    \Res_{Y=0}\frac{e^{k Y}}{(e^Y-1)^r(e^{X-Y}-1)}
    =-
    \Res_{Y=X}\frac{e^{k Y}}{(e^Y-1)^r(e^{X-Y}-1)}
    \\
    +
    \frac{1}{2\pi i}\int_{-R+(2\pi+c) i}^{-R+ci}
    \frac{e^{k Y}dY}{(e^Y-1)^r(e^{X-Y}-1)}
    +
    \frac{1}{2\pi i}\int_{R+ci}^{R+(2\pi+c) i}
    \frac{e^{k Y}dY}{(e^Y-1)^r(e^{X-Y}-1)}.
  \end{multline}
  We have uniformly
  \begin{equation}
    \frac{e^{k Y}}{(e^Y-1)^r(e^{X-Y}-1)}=
    \begin{cases}
      O(e^{(k+1)Y})\qquad &(\Re Y\to -\infty),\\      
      O(e^{(k-r)Y})\qquad &(\Re Y\to \infty).
    \end{cases}
  \end{equation}
  Since $0\leq k\leq r-1$, by taking the limit $R\to\infty$, we obtain
  \begin{equation}
    \begin{split}
      \Res_{Y=0}\frac{e^{k Y}}{(e^Y-1)^r(e^{X-Y}-1)}
      &=-
      \Res_{Y=X}\frac{e^{k Y}}{(e^Y-1)^r(e^{X-Y}-1)}
      \\
      &=
      \frac{e^{k X}}{(e^X-1)^r}.
    \end{split}
  \end{equation}
\end{proof}
\begin{lemma}
For $0\leq z<1$, we have
  \label{lm:vanish}
  \begin{equation}
    \Res_{\eta=0}
    \bigl(
    \mathfrak{F}(\xi;z;\omega_2)^r\eta^{-1}
    -
    \mathfrak{F}(\eta;z;\omega_2)^r
    \mathfrak{F}(\xi-\eta;\{rz\};\omega_2)
    \bigr)=0.
  \end{equation}
\end{lemma}
\begin{proof}
  We have
  \begin{equation}
    \begin{split}
      \mathfrak{F}(\eta)^r
      \mathfrak{F}(\xi-\eta;\{rz\};\omega_2)
      &=
      \Bigl(\frac{2\pi i}{\omega_2}\Bigr)^{r+1}
      \Bigl(\frac{e^{2\pi i\eta z/\omega_2}}{e^{2\pi i\eta/\omega_2}-1}\Bigr)^r
      \frac{e^{2\pi i(\xi-\eta) \{rz\}/\omega_2}}{e^{2\pi i(\xi-\eta)/\omega_2}-1}
      \\
      &=
      \Bigl(\frac{2\pi i}{\omega_2}\Bigr)^{r+1}
      e^{2\pi i\xi\{rz\}/\omega_2}
      \frac{e^{2\pi i\eta [rz]/\omega_2}}{(e^{2\pi i\eta/\omega_2}-1)^r(e^{2\pi i(\xi-\eta)/\omega_2}-1)}.
    \end{split}
  \end{equation}
  Since $0\leq z<1$, we see that $0\leq [rz]\leq r-1$.
  By Lemma \ref{lm:Laurent}, we obtain
  \begin{equation}
    \begin{split}
      \Res_{\eta=0}    \mathfrak{F}(\eta)^r
      \mathfrak{F}(\xi-\eta;\{rz\};\omega_2)
      &=
      \Bigl(\frac{2\pi i}{\omega_2}\Bigr)^r
      e^{2\pi i\xi\{rz\}/\omega_2}
      \frac{e^{2\pi i\xi [rz]/\omega_2}}{(e^{2\pi i\xi/\omega_2}-1)^r}
      \\
      &=
      \Bigl(\frac{2\pi i}{\omega_2}\Bigr)^r
      \frac{e^{2\pi i\xi rz/\omega_2}}{(e^{2\pi i\xi/\omega_2}-1)^r}
      \\
      &=\mathfrak{F}(\xi)^r.
    \end{split}
  \end{equation}
\end{proof}
Assume that $0\leq z<1$ and $x\neq 0$, and fix $\xi$.
By Lemma \ref{lm:vanish}, we see that
\begin{multline}
  \mathfrak{K}_r(\xi;x,0,z;\omega_1,\omega_2)
  \\
  \begin{aligned}
    &=
  \mathfrak{K}_r(\xi;x,0,z;\omega_1,\omega_2)
    \\
    &\qquad
    -\mathcal{H}_1(x,0;\omega_1,\omega_2)
    \Res_{\eta=0}
    \bigl(
    \mathfrak{F}(\xi;z;\omega_2)^r\eta^{-1}
    -
    \mathfrak{F}(\eta;z;\omega_2)^r
    \mathfrak{F}(\xi-\eta;\{rz\};\omega_2)
    \bigr)
    \\
    &=
    \Res_{\eta=0}
    \bigl(
    \widetilde{\mathfrak{E}}(\xi;x;\omega_1,\omega_2)\mathfrak{F}(\xi;z;\omega_2)^r\eta^{-1}
    -
    \widetilde{\mathfrak{E}}(\eta;x;\omega_1,\omega_2)\mathfrak{F}(\eta;z;\omega_2)^r
    \mathfrak{F}(\xi-\eta;\{rz\};\omega_2)
    \bigr),
  \end{aligned}
\end{multline}
where the second equality follows from the definitions \eqref{eq:def_Ft} and
\eqref{K_r}. 
Now it is easy to take the limit $x\to 0$. Using Lemma \ref{lem2-1}, 
we have
\begin{multline}
  \label{eq:K_F1}
    \lim_{x\to0}\mathfrak{K}_r(\xi;x,0,z;\omega_1,\omega_2)
    \\
    \begin{aligned}
      &=\lim_{x\to0}\Res_{\eta=0}
    \bigl(
    \widetilde{\mathfrak{E}}(\xi;x;\omega_1,\omega_2)\mathfrak{F}(\xi;z;\omega_2)^r\eta^{-1}
    -
    \widetilde{\mathfrak{E}}(\eta;x;\omega_1,\omega_2)\mathfrak{F}(\eta;z;\omega_2)^r
    \mathfrak{F}(\xi-\eta;\{rz\};\omega_2)
    \bigr)
    \\
    &=
    \Res_{\eta=0}
    \bigl(
    \widetilde{\mathfrak{E}}(\xi;0;\omega_1,\omega_2)\mathfrak{F}(\xi;z;\omega_2)^r\eta^{-1}
    -
    \widetilde{\mathfrak{E}}(\eta;0;\omega_1,\omega_2)\mathfrak{F}(\eta;z;\omega_2)^r
    \mathfrak{F}(\xi-\eta;\{rz\};\omega_2)
    \bigr)
    \\
    &=
    \mathfrak{K}_r(\xi;0,0,z;\omega_1,\omega_2),
  \end{aligned}
\end{multline}
which implies the continuity of $\mathfrak{K}_r$ with respect to $x$ at $x=0$.

Here we remark that the situation in the case $z=1$ is slightly different. 
By
\begin{equation}
  \label{eq:Fz1} 
  \mathfrak{E}(\xi;x,y;\omega_1,\omega_2)
  =-\mathfrak{E}(-\xi;-x,-y;\omega_1,\omega_2)
\end{equation}
and \eqref{eq:Gz1}, we have
\begin{equation}
  \begin{split}
    \mathfrak{D}_r(\xi;x,y,1;\omega_1,\omega_2)
    &=\mathfrak{E}(\xi;x,y;\omega_1,\omega_2) \mathfrak{F}(\xi;1;\omega_2)^r
    \\
    &=(-1)^r \mathfrak{E}(\xi;x,y;\omega_1,\omega_2) \mathfrak{F}(-\xi;0;\omega_2)^r
    \\
    &=(-1)^{r+1} \mathfrak{E}(-\xi;-x,-y;\omega_1,\omega_2) \mathfrak{F}(-\xi;0;\omega_2)^r
    \\
    &=(-1)^{r+1} \mathfrak{D}_r(-\xi;-x,-y,0;\omega_1,\omega_2),
  \end{split}
\end{equation}
and so for $x\neq 0$,
\begin{multline}
\label{eq:K_z_01}
  \mathfrak{K}_r(\xi;x,0,1;\omega_1,\omega_2)
  \\
  \begin{aligned}
    &=
    \Res_{\eta=0}
    \bigl(
    \mathfrak{D}_r(\xi;x,0,1;\omega_1,\omega_2)
    \eta^{-1}-
    \mathfrak{D}_r(\eta;x,0,1;\omega_1,\omega_2)
    \mathfrak{F}(\xi-\eta;0;\omega_2)
    \bigr)
    \\
    &=(-1)^{r+1}
    \Res_{\eta=0}
    \bigl(
    \mathfrak{D}_r(-\xi;-x,0,0;\omega_1,\omega_2)
    \eta^{-1}
    \\
    &\qquad\qquad
    +
    \mathfrak{D}_r(-\eta;-x,0,0;\omega_1,\omega_2)
    \mathfrak{F}(-\xi+\eta;1;\omega_2)
    \bigr)
    \\
    &=(-1)^{r+1}
    \Res_{\eta=0}
    \bigl(
    \mathfrak{D}_r(-\xi;-x,0,0;\omega_1,\omega_2)
    \eta^{-1}
    \\
    &\qquad\qquad
    +
    \mathfrak{D}_r(-\eta;-x,0,0;\omega_1,\omega_2)
    \mathfrak{F}(-\xi+\eta;0;\omega_2)
    e^{2\pi i(-\xi+\eta)/\omega_2}
    \bigr)
    \\
    &=
    (-1)^{r+1}
    \mathfrak{K}_r(-\xi;-x,0,0;\omega_1,\omega_2)
    \\
    &\qquad\qquad
    +(-1)^{r}
    \Res_{\eta=0}
    \bigl(
    \mathfrak{D}_r(\eta;-x,0,0;\omega_1,\omega_2)
    \mathfrak{F}(-\xi-\eta;0;\omega_2)
    (e^{2\pi i(-\xi-\eta)/\omega_2}-1)
    \bigr)
    \\
    &=
    (-1)^{r+1}
    \mathfrak{K}_r(-\xi;-x,0,0;\omega_1,\omega_2)
    +R_r(x),
  \end{aligned}
\end{multline}
where 
\begin{equation}
  R_r(x)=
  (-1)^{r}
  \frac{2\pi i}{\omega_2}
  \Res_{\eta=0}
  \mathfrak{D}_r(\eta;-x,0,0;\omega_1,\omega_2).
\end{equation}
Therefore \eqref{eq:K_F1} and \eqref{eq:K_z_01} imply
\begin{equation}
\label{eq:K_z_02}
\begin{split}
  \lim_{x\to 0}(\mathfrak{K}_r(\xi;x,0,1;\omega_1,\omega_2)-R_r(x))
    &=
    \lim_{x\to 0}(-1)^{r+1}
    \mathfrak{K}_r(-\xi;-x,0,0;\omega_1,\omega_2)
    \\
    &=
    (-1)^{r+1}\mathfrak{K}_r(-\xi;0,0,0;\omega_1,\omega_2).
    \\
    &=
    \mathfrak{K}_r(\xi;0,0,1;\omega_1,\omega_2).
  \end{split}
\end{equation}
Note that
\eqref{eq:K_z_02} means that 
if we subtract
the constant term 
$R_r(x)$ with respect to $\xi$
from $\mathfrak{K}_r(\xi;x,0,1;\omega_1,\omega_2)$,
then it goes to
$\mathfrak{K}_r(\xi;0,0,1;\omega_1,\omega_2)$
when $x\to 0$.



\section{Proof of Theorem \ref{Main-theorem}} \label{sec-3-2}

For
$M,N\in\mathbb{N}$, let $C_{M,N}$ be the boundary of the parallelogram
whose vertices consist of $\pm(M+1/2)\omega_1\pm(N+1/2)\omega_2$.
We show some key lemmas about the following limit:
\begin{equation}
  \label{eq:main_lim}
  \Lim_{M,N}
  \int_{C_{M,N}}
  \xi^{-k}\mathfrak{E}(\xi)\mathfrak{F}(\xi)^rd\xi
  =0.
\end{equation}

\begin{lemma}
\label{lm:bound1}
  Let $r\in\mathbb{N}$.
  Assume that $k\in\mathbb{N}$ with $k\geq2$, $(x,y)\neq(0,0)$ and $0\leq z\leq1$.
Then \eqref{eq:main_lim} holds.
\end{lemma}

\begin{lemma}\label{lm:bound2}
  Let $r\in\mathbb{N}$.
  Assume that $k=1$, $(x,y)\neq(0,0)$, $0<z<1$, $y+rz\not\in\mathbb{Z}$.
  Then \eqref{eq:main_lim} holds.
\end{lemma}
\begin{remark}
  In Lemma \ref{lm:bound2}, if $z=0,1$ and $y+rz\not\in\mathbb{Z}$,
  then $\lim_{M\to\infty}\lim_{N\to\infty}$ is allowed while
  if
  $z\neq 0,1$ and $y+rz\in\mathbb{Z}$,
  then $\lim_{N\to\infty}\lim_{M\to\infty}$ is allowed.
\end{remark}
\begin{lemma}\label{lm:bound3}
  Let $r\in\mathbb{N}$.
  Assume that $k\in\mathbb{N}$ with $k\geq2$, $(x,y)=(0,0)$ and $0\leq z\leq 1$.
  If $k=2$, then assume $z\neq 0,1$.
  Then \eqref{eq:main_lim} holds.
\end{lemma}
\begin{lemma}\label{lm:bound4}
  Let $r\in\mathbb{N}$.
  Assume that $k=1$, $(x,y)=(0,0)$, $0<z<1$ and $rz\not\in\mathbb{Z}$.
  Then \eqref{eq:main_lim} holds.
\end{lemma}


To show these lemmas, we need some preparation.
Let $C_{M,N}^1$ be the segment from
$-(M+1/2)\omega_1+(N+1/2)\omega_2$ 
to
$(M+1/2)\omega_1+(N+1/2)\omega_2$
and $C_{M,N}^2$, from
$(M+1/2)\omega_1+(N+1/2)\omega_2$ 
to
$(M+1/2)\omega_1-(N+1/2)\omega_2$.

Let $C_N^1$ be the line from
$-\infty\omega_1+(N+1/2)\omega_2$ 
to
$\infty\omega_1+(N+1/2)\omega_2$
and
 $C_M^2$, from
$(M+1/2)\omega_1-\infty\omega_2$ 
to
$(M+1/2)\omega_1+\infty\omega_2$.
\begin{lemma}
\label{lm:bound_G}
  Let $0<\alpha<1$. 
  Then there exist positive constants $c,\delta$, depending only on $\alpha$,
such that
  for all
 $\xi\in\bigcup_{M,N\in\mathbb{N}}C_{M,N}$,
  \begin{equation}
    \biggl\lvert
    \frac{e^{2\pi i\xi \alpha/\omega_2}}{e^{2\pi i\xi/\omega_2}-1}
    \biggr\rvert
    \leq ce^{-\delta\abs{\Re (i\xi/\omega_2)}}.
  \end{equation}
\end{lemma}

\begin{proof}
This is because the left-hand side is periodic in $\xi$ with period $\omega_2$,
while, since $0<\alpha<1$, the left-hand side is of exponential decay when $\xi$
moves along $C_N^1$.
\end{proof}
 
\begin{lemma}
  \label{lm:Phi}
  For $\alpha\in\mathbb{R}\setminus\mathbb{Z}$ and $\beta\in\mathbb{C}\setminus\mathbb{Z}$,
  \begin{equation}
    \label{eq:lim1}
    \begin{split}
      \lim_{N\to\infty}
      \sum_{n=-N}^N 
      \frac{e^{2\pi in \alpha}}{\beta+n}
      =
      2 \pi i 
      \frac{e^{2 \pi i  \beta (1-\{\alpha\})}}{e^{2 \pi i \beta }-1}
      (=\Phi(\alpha,\beta),{\rm say}).
    \end{split}
  \end{equation}
\end{lemma}
\begin{proof}
  In this proof, we temporarily set $(\omega_1,\omega_2)=(1,i)$.
  Assume $0<\alpha<1$. Let
  \begin{equation}
    f(\xi)=\frac{e^{2 \pi \xi\alpha}}{e^{2 \pi \xi }-1}.
  \end{equation}
  Then by Lemma \ref{lm:bound_G}, we see that
  there exist $c,\delta>0$ such that
  \begin{equation}
    \abs{f(\xi)}    \leq ce^{-\delta\abs{\Re \xi}}
  \end{equation}
for all
 $\xi\in\bigcup_{M,N\in\mathbb{N}}C_{M,N}$.
Putting $\xi=a+(N+1/2)i$ on $C_N^1$, for all $N$ with sufficiently large $\abs{N}$,
  we obtain
  \begin{equation}
  \abs{(\xi+i\beta)^{-1}f(\xi)}
  \leq
  c
  \abs{N+1/2+\Re\beta}^{-1}
  e^{-\delta\abs{a}}
  \leq
  e^{-\delta\abs{a}}
\end{equation}
and
\begin{equation}
  \int_{-\infty}^{\infty}e^{-\delta\abs{a}}
  da<\infty.
\end{equation}
Hence 
\begin{align}
  \lim_{M\to\infty}
  \int_{C_{M,N}^1}(\xi+i\beta)^{-1}f(\xi)d\xi
  &=
  \int_{C_N^1}(\xi+i\beta)^{-1}f(\xi)d\xi,
\end{align}
\begin{align}\label{koremotsuika}
  \lim_{N\to\pm\infty} \int_{C_N^1}(\xi+i\beta)^{-1}f(\xi)d\xi
  &=0.
\end{align}
Setting $\xi=a+bi$ on $C_{M,N}^2$ with $a=M+1/2$,
for all $M$ with sufficiently large $\abs{M}$, we have
\begin{equation}
  \begin{split}
  \biggl\lvert
  \int_{C_{M,N}^2}
  (\xi+i\beta)^{-1}f(\xi)d\xi
  \biggr\rvert
  &\leq
  ce^{-\delta\abs{a}}
  \int_{-(N+1/2)}^{N+1/2}\abs{(a+bi+\beta i)^{-1}}
  db
  \\
  &\leq
  ce^{-\delta\abs{a}}
  \abs{a-\Im\beta}^{-1}
  (2N+1),
\end{split}
\end{equation}
which implies that
\begin{equation}
  \lim_{M\to\infty}
  \int_{C_{M,N}^2}
  (\xi+i\beta)^{-1}f(\xi)d\xi
  =0.
\end{equation}
Hence for a sufficiently large $N>0$, we have 
\begin{equation}\label{koremokoremo}
  \begin{split}
    \lim_{M\to\infty}  \int_{C_{M,N}}(\xi+i\beta)^{-1}f(\xi)d\xi
    &=
    \biggl(
    \int_{C_N^1}-
    \int_{C_{-N-1}^1}
    \biggr)
(\xi+i\beta)^{-1}f(\xi)d\xi
\\
    &=
    \sum_{n=-N}^N
    \frac{e^{2\pi in \alpha}}{\beta+n}
    +
    2 \pi i 
    \frac{e^{-2 \pi i \beta\alpha}}{e^{-2 \pi i \beta }-1},
  \end{split}
\end{equation}
where the second equality follows by counting the residues of the poles between
$C_N^1$ and $C_{-N-1}^1$.   But \eqref{koremotsuika} implies that the second member
of \eqref{koremokoremo} tends to 0 as $N\to\infty$.
Hence the third member also tends to 0, which implies
\eqref{eq:lim1}.


\end{proof}

\begin{lemma}
\label{lm:ps}
Let $\alpha\in\mathbb{R}\setminus\mathbb{Z}$. Then
there exists $c=c(\alpha)>0$ such that 
\begin{equation}
    \biggl\lvert
    \sum_{n=-N}^N  \frac{e^{2\pi i n\alpha}}{\beta+n}
    \biggr\rvert
    \leq \frac{c}{\abs{\Im\beta}}
\end{equation}
for all $N\in\mathbb{N}$ and $\beta\in\mathbb{C}\setminus\mathbb{R}$.
\end{lemma}
\begin{proof}
We have
\begin{equation}
  \begin{split}
    \sum_{n=-N}^N  \frac{e^{2\pi in\alpha}}{\beta+n}
    &=
    \biggl(\sum_{n=-N}^N  e^{2\pi in\alpha}\biggr)\frac{1}{\beta+N}
    -\int_{-N}^N
    \biggl(
    \sum_{n=-N}^{[\xi]}  e^{2\pi in\alpha}
    \biggr)
    \frac{-d\xi}{(\beta+\xi)^{2}}
    \\
    &=\frac{e^{-2\pi iN\alpha}(1-e^{2\pi i\alpha (2N+1)})}{1-e^{2\pi i \alpha}}
    \frac{1}{\beta+N}
    \\
    &\qquad\qquad
    -\int_{-N}^N
    \frac{e^{-2\pi iN\alpha}(1-e^{2\pi i\alpha ([\xi]+N+1)})}{1-e^{2\pi i \alpha}}
    \frac{-d\xi}{(\beta+\xi)^2}.
  \end{split}
\end{equation}
Hence
\begin{equation}
  \begin{split}
    \biggl\lvert
    \sum_{n=-N}^N  \frac{e^{2\pi in\alpha}}{\beta+n}
    \biggr\rvert
    &\leq\frac{2}{\abs{1-e^{2\pi i \alpha}}}
    \biggl(
    \frac{1}{\abs{\beta+N}}
    +
    \int_{-N}^N
    \frac{d\xi}{\abs{\beta+\xi}^2}
    \biggr)\\
    &\leq\frac{2}{\abs{1-e^{2\pi i \alpha}}}
    \biggl(
    \frac{1}{\abs{\Im\beta}}
    +
    \int_{-\infty}^\infty
    \frac{d\xi}{\abs{i \Im\beta+\xi}^2}
    \biggr)
    \\
    &=\frac{2}{\abs{1-e^{2\pi i \alpha}}}
    \biggl(
    \frac{1}{\abs{\Im\beta}}
    +
    \frac{1}{\abs{\Im\beta}}
    \int_{-\infty}^\infty
    \frac{d\xi'}{\abs{i+\xi'}^2}
    \biggr)
    \\
    &\leq \frac{c}{\abs{\Im\beta}}
  \end{split}
\end{equation}
for some $c>0$.
\end{proof}

\begin{lemma}
  \label{lm:est}
  There exists $c>0$, depending on $\omega_1,\omega_2$, such that 
  \begin{equation}
    \abs{\xi}>c(a^2+b^2)^{1/2}
  \end{equation}
for all $\xi=a\omega_1+b\omega_2\in \mathbb{C}\setminus\{0\}$.
\end{lemma}

\begin{proof}[Proof of Lemma \ref{lm:bound1}]
  We first show that on $\bigcup_{M,N\in\mathbb{N}}C_{M,N}$,
  \begin{equation}
    \abs{\mathfrak{E}(\xi)\mathfrak{F}(\xi)^r} \leq c_0
  \end{equation}
  for some $c_0>0$.  In fact, the boundedness of $\mathfrak{E}(\xi)$ follows from
\eqref{2-1}, \eqref{2-2}, while the boundedness of $\mathfrak{F}(\xi)$ follows
from Lemma \ref{lm:bound_G} when $0<z<1$, directly when $z=0,1$.
 
Consider the path $C_{M,N}^1$. 
By Lemma \ref{lm:est}, we have
  \begin{equation}
    \abs{\xi^{-k}\mathfrak{E}(\xi)\mathfrak{F}(\xi)^r}
    \leq 
    \frac{c_0c}{(a^2+b^2)^{k/2}}.
  \end{equation}
Since
  \begin{equation}
    \int_{-\infty}^\infty \frac{da}{(a^2+1)^{k/2}}<\infty,
  \end{equation}
we have
  \begin{equation}
    \Lim_{M,N}
    \int_{C_{M,N}^1}
    \xi^{-k}\mathfrak{E}(\xi)\mathfrak{F}(\xi)^rd\xi
    =0.
  \end{equation}
This together with similar estimations for the rest of the path $C_{M,N}$
implies \eqref{eq:main_lim}.
\end{proof}

\begin{proof}[Proof of Lemma \ref{lm:bound2}]
  First we consider the path $C_{M,N}^1$.
  For $\xi=a\omega_1+b\omega_2$ with $a\in\mathbb{R}$ and $b=N+1/2>1$,
  \begin{equation}
    \label{eq:b2}
    \abs{\xi^{-1}\mathfrak{E}(\xi)\mathfrak{F}(\xi)^r}
    \leq c_1
    \frac{e^{-\abs{a} r\delta\Re (i/\tau)}}{(a^2+b^2)^{1/2}}
  \end{equation}
for some $c_1>0$
by Lemmas \ref{lm:bound_G} (which can be applied because $0<z<1$) and \ref{lm:est}, 
which implies
  \begin{equation}
    \Lim_{M,N}
    \int_{C_{M,N}^1}
    \xi^{-1}\mathfrak{E}(\xi)\mathfrak{F}(\xi)^rd\xi
    =0.
  \end{equation}
  
  Next we consider the path $C_{M,N}^2$.
  We have, using \eqref{2-2} and \eqref{eq:perio}, 
  \begin{multline}\label{4-24}
    \int_{C_{M,N}^2}
    \xi^{-1}\mathfrak{E}(\xi)\mathfrak{F}(\xi)^rd\xi
    \\
    \begin{aligned}
      &= 
      \int_{-(N+1/2)}^{N+1/2}
      (a\omega_1+b\omega_2)^{-1}
      \mathfrak{E}(a\omega_1+b\omega_2)
      \mathfrak{F}(a\omega_1+b\omega_2)^r
      (-\omega_2db)
      \\
      &= 
      \sum_{n=-N}^N
      \int_{-1/2}^{1/2}
      (a\omega_1+(b+n)\omega_2)^{-1}
      \mathfrak{E}(a\omega_1+b\omega_2)
      \mathfrak{F}(a\omega_1+b\omega_2)^r
      e^{2\pi i(y+rz)n}
      (-\omega_2db)
      \\
      &= 
      \int_{-1/2}^{1/2}
      \mathfrak{E}(a\omega_1+b\omega_2)
      \mathfrak{F}(a\omega_1+b\omega_2)^r
      \sum_{n=-N}^N
      \frac{e^{2\pi i(y+rz)n}}
      {(a\omega_1+(b+n)\omega_2)}
      (-\omega_2db),
    \end{aligned}
  \end{multline}
  where $a=M+1/2$. Let $h(M,N,b)$ be the integrand of the last member.
  Then
  Lemmas \ref{lm:bound_G} and \ref{lm:ps} imply
  \begin{equation}
    \abs{h(M,N,b)}\leq \frac{c_2    e^{-\abs{a} r\delta\Re (i/\tau)}}{\abs{a\Im \tau^{-1}}}
  \end{equation}
  for some $c_2>0$, which implies
  \begin{equation}
    \lim_{\substack{M\to\infty\\N\to\infty}}
    \int_{C_{M,N}^2}
    \xi^{-1}\mathfrak{E}(\xi)\mathfrak{F}(\xi)^rd\xi
    =
    \lim_{N\to\infty}\lim_{M\to\infty}
    \int_{C_{M,N}^2}
    \xi^{-1}\mathfrak{E}(\xi)\mathfrak{F}(\xi)^rd\xi=0.
  \end{equation}
On the other hand, by Lemma \ref{lm:Phi} and \eqref{4-24} we have
  \begin{multline}
    \lim_{N\to\infty}
    \int_{C_{M,N}^2}
    \xi^{-1}\mathfrak{E}(\xi)\mathfrak{F}(\xi)^rd\xi
    \\
    =
    \int_{-1/2}^{1/2}
    \mathfrak{E}(a\omega_1+b\omega_2)
    \mathfrak{F}(a\omega_1+b\omega_2)^r
    \Phi(y+rz,a/\tau+b)
    (-db).
  \end{multline}
Applying Lemma \ref{lm:bound_G} (available because $y+rz\notin\mathbb{Z}$) to 
$\Phi$, we obtain
  for some $c_3>0$,
  \begin{equation}
    \abs{\mathfrak{E}(a\omega_1+b\omega_2)
      \mathfrak{F}(a\omega_1+b\omega_2)^r
      \Phi(y+rz,a/\tau+b)}
    \leq c_3 e^{-\abs{a} r\delta\Re (i/\tau)}
    e^{-\abs{a} \delta'\Re (i/\tau)},
\end{equation}
which yields
  \begin{equation}
    \lim_{M\to\infty}
    \int_{-1/2}^{1/2}
    \mathfrak{E}(a\omega_1+b\omega_2)
    \mathfrak{F}(a\omega_1+b\omega_2)^r
    \Phi(y+rz,a/\tau+b)
    (-db)=0.
  \end{equation}
  
Similar estimations also hold for the rest of the path $C_{M,N}$, therefore
\eqref{eq:main_lim} follows.
\end{proof}

\begin{lemma}
\label{lm:D00}
Let $(x,y)=(0,0)$.
Then on $\bigcup_{M,N\in\mathbb{N}}C_{M,N}$, we have
\begin{equation}
\label{eq:D00}
  \abs{\mathfrak{E}(\xi)\mathfrak{F}(\xi)^r} \leq 
  \begin{cases}
c_4 e^{-\delta|\Re (i\xi/\omega_2)|}\qquad&(0<z<1) \\
c_4 (1+\abs{\Re (i\xi/\omega_2)})\qquad&(z=0,1) 
\end{cases}
\end{equation}
for some $c_4,\delta>0$.
\end{lemma}
\begin{proof}
For $\xi\in\bigcup_{M,N\in\mathbb{N}}C_{M,N}$, put
$\xi=(a+m)\omega_1+(b+n)\omega_2$ so that $0\leq a,b<1$.
Then for some $c,c'>0$,
\begin{equation}
  \begin{split}
    \abs{\mathfrak{E}(\xi)\mathfrak{F}(\xi)^r}
    &=
    \abs{\mathfrak{E}(a\omega_1+b\omega_2)+2\pi im/\omega_2}
    \abs{\mathfrak{F}(\xi)^r}
    \\
    &\leq (c+c'|\Re (i\xi/\omega_2)|)
    \abs{\mathfrak{F}(\xi)^r}
  \end{split}
\end{equation}
by \eqref{eq:pE00-1} and \eqref{eq:pE00-2}.
If $0<z<1$, then by Lemma \ref{lm:bound_G} and replacing $\delta$ by a slightly small one, we obtain \eqref{eq:D00}.
If $z=0,1$, then $\mathfrak{F}(\xi)$ is bounded and
\eqref{eq:D00} holds.
\end{proof}
\begin{proof}[Proof of Lemma \ref{lm:bound3}]
By Lemma \ref{lm:D00}, if $0<z<1$, we see that
 on $\bigcup_{M,N\in\mathbb{N}}C_{M,N}$, 
\begin{equation}
  \abs{\mathfrak{E}(\xi)\mathfrak{F}(\xi)^r} \leq c_5
\end{equation}
for some $c_5>0$. If $z=0,1$, then 
\begin{equation}
  \abs{\xi^{-1}\mathfrak{E}(\xi)\mathfrak{F}(\xi)^r} \leq c_5
\end{equation}
holds instead. By the same argument as in the proof of Lemma \ref{lm:bound1},
we have \eqref{eq:main_lim} in these cases.
\end{proof}
\begin{proof}[Proof of Lemma \ref{lm:bound4}]
Since the same estimation as
\eqref{eq:b2} holds,
the proof of Lemma \ref{lm:bound2} works well in this case.
\end{proof}

\begin{proof}[Proof of Theorem \ref{Main-theorem}]
  Let
  $k\in\mathbb{N}$.
  Due to the properties of $\mathfrak{E}$ and $\mathfrak{F}$,
  we see that 
  $\xi^{-k}\mathfrak{D}_r(\xi)$ has poles only at
  $\xi=m\omega_1+n\omega_2$ ($m,n\in\mathbb{Z}$). 
  Hence
  \begin{equation}
    \frac{1}{2\pi i}\int_{C_{M,N}}\xi^{-k}\mathfrak{D}_r(\xi)d\xi=
    \sum_{\substack{-M\leq m\leq M\\-N\leq n\leq N}}
    \Res_{\xi=m\omega_1+n\omega_2}\xi^{-k}\mathfrak{D}_r(\xi).
  \end{equation}
  By Lemmas  \ref{lm:bound1}, \ref{lm:bound2},
 \ref{lm:bound3} and  \ref{lm:bound4}, we have
  \begin{equation}\label{4-31}
    \Lim_{M,N}
    \sum_{\substack{-M\leq m\leq M\\-N\leq n\leq N}}
    \Res_{\xi=m\omega_1+n\omega_2}\xi^{-k}\mathfrak{D}_r(\xi)=
    \frac{1}{2\pi i}
    \Lim_{M,N}
    \int_{C_{M,N}}\xi^{-k}\mathfrak{D}_r(\xi)d\xi=
    0
  \end{equation}
under the assumptions of Theorem \ref{Main-theorem}.
 
  The poles at
  $\xi=m\omega_1+n\omega_2$ ($m\in\mathbb{Z}\setminus\{0\}$, $n\in\mathbb{Z}$) are simple and their residues are calculated as
  \begin{equation}\label{4-32}
    \begin{split}
      \Res_{\xi=m\omega_1+n\omega_2}\xi^{-k}\mathfrak{D}_r(\xi)
      &=\Res_{h=0}(m\omega_1+n\omega_2+h)^{-k}\mathfrak{E}(m\omega_1+n\omega_2+h)\mathfrak{F}(m\omega_1+n\omega_2+h)^r\\
       &=\Res_{h=0}(m\omega_1+n\omega_2+h)^{-k}\mathfrak{E}(h)\mathfrak{F}(m\omega_1+h)^r e^{2\pi i (mx+n(y+rz))}\\
      &=
      (m\omega_1+n\omega_2)^{-k}
      \mathfrak{F}(m\omega_1)^r
      e^{2\pi i (mx+n(y+rz))}.
    \end{split}
  \end{equation}
  Similarly the poles at
  $\xi=n\omega_2$ ($n\in\mathbb{Z}\setminus\{0\}$) are of order $(1+r)$ and their residues are
  \begin{equation}\label{4-33}
    \begin{split}
      \Res_{\xi=n\omega_2}\xi^{-k}\mathfrak{D}_r(\xi)
      &=
      \Res_{h=0}(n\omega_2+h)^{-k}\mathfrak{D}_r(h)e^{2\pi i n(y+rz)}\\
      &=
      \Res_{h=0}
      \Bigl(
      \sum_{j=0}^\infty
      (-1)^j\frac{(k+j-1)!}{j!(k-1)!}(n\omega_2)^{-k-j}h^j
      \Bigr)
      \mathfrak{D}_r(h)e^{2\pi i n(y+rz)}.
    \end{split}
  \end{equation}
  Since $\mathfrak{D}_r(\xi)$ has a pole of order $(1+r)$ at the origin,
  by putting $\mathfrak{D}_r(\xi)=\sum_{l=0}^\infty \mathcal{D}_l\xi^{l-r-1}$ we obtain
  \begin{equation}\label{4-34}
    \Res_{\xi=n\omega_2}\xi^{-k}\mathfrak{D}_r(\xi)
    =
    \Bigl(
    \sum_{j=0}^r
    \mathcal{D}_{r-j}
    \frac{(-1)^j}{j!}
    \frac{(k+j-1)!}{(k-1)!}
    (n\omega_2)^{-k-j}e^{2\pi i n(y+rz)}
    \Bigr).
  \end{equation}
Hence by Proposition \ref{prop:bernoulli}, with noting $y+rz\not\in\mathbb{Z}$ 
in the case $k=1$,
  we have 
  \begin{multline}\label{4-35}
    \lim_{N\to\infty}\sum_{\substack{-N\leq n\leq N\\n\neq 0}}
    \Res_{\xi=n\omega_2}\xi^{-k}\mathfrak{D}_r(\xi)
    \\
    \begin{aligned}
      &=
      \sum_{j=0}^r
      \mathcal{D}_{r-j}
      \frac{(-1)^j}{j!}
      \frac{(k+j-1)!}{(k-1)!}
      \lim_{N\to\infty}\sum_{\substack{-N\leq n\leq N\\n\neq 0}}
      (n\omega_2)^{-k-j}e^{2\pi i n(y+rz)}
      \\
      &=
      -\sum_{j=0}^r
      \mathcal{D}_{r-j}
      \frac{(-1)^j}{j!}
      \frac{(k+j-1)!}{(k-1)!}
      \frac{\mathcal{B}_{k+j}(\{y+rz\};\omega_2)}{(k+j)!}.
    \end{aligned}
  \end{multline}
Since from \eqref{bernoulli} we see that
  \begin{equation}\label{4-36}
    \begin{split}
      \frac{d^j}{d\xi^j}(\xi^{-1}-\mathfrak{F}(\xi;z;\omega_2))
      &=
      -\sum_{l=j+1}^\infty \frac{(l-1)!}{(l-j-1)!}\frac{\mathcal{B}_l(z;\omega_2)}
{l!}\xi^{l-j-1}
      \\
      &=
      -\sum_{l=1}^\infty \frac{(l+j-1)!}{(l-1)!}\frac{\mathcal{B}_{l+j}(z;\omega_2)}
{(l+j)!}\xi^{l-1},
    \end{split}
  \end{equation}
  and for $0\leq j\leq r$
  \begin{equation}\label{4-37}
    \Res_{\xi=0}\xi^{-k}\frac{d^j}{d\xi^j}\xi^{-1}=0  ,
  \end{equation}
  we have
  \begin{equation}\label{4-38}
    \begin{split}
    \lim_{N\to\infty}\sum_{\substack{-N\leq n\leq N\\n\neq 0}}
      \Res_{\xi=n\omega_2}\xi^{-k}\mathfrak{D}_r(\xi)
      &=
      \Res_{\xi=0}\xi^{-k}
      \sum_{j=0}^r
      \mathcal{D}_{r-j}
      \frac{(-1)^j}{j!}
      \frac{d^j}{d\xi^j}(\xi^{-1}-\mathfrak{F}(\xi;\{y+rz\};\omega_2))
      \\
      &=
      -   \Res_{\xi=0}\xi^{-k}
      \sum_{j=0}^r
      \mathcal{D}_{r-j}
      \frac{(-1)^j}{j!}
      \frac{d^j}{d\xi^j}\mathfrak{F}(\xi;\{y+rz\};\omega_2)
      \\
      &=
      -\Res_{\xi=0}\xi^{-k}
      \Res_{\eta=0}
      \mathfrak{D}_r(\eta)
      \mathfrak{F}(\xi-\eta;\{y+rz\};\omega_2).
    \end{split}
  \end{equation}
  
From \eqref{4-31}, \eqref{4-32} and \eqref{4-38} we obtain
  \begin{equation}\label{4-39}
    \begin{split}
         &0=\Lim_{M,N}
      \sum_{\substack{-M\leq m\leq M\\-N\leq n\leq N}}
      \Res_{\xi=m\omega_1+n\omega_2}\xi^{-k}\mathfrak{D}_r(\xi)
      \\
      &=\Lim_{M,N}
      \biggl(
      \sum_{\substack{-M\leq m\leq M\\-N\leq n\leq N\\m\neq 0}}
      \frac{\mathfrak{F}(m\omega_1)^re^{2\pi i(mx+n(y+rz))}}{(m\omega_1+n\omega_2)^{k}}
      +
      \Res_{\xi=0}\xi^{-k}\mathfrak{D}_r(\xi)
      +
      \sum_{\substack{-N\leq n\leq N\\n\neq 0}}
      \Res_{\xi=n\omega_2}\xi^{-k}\mathfrak{D}_r(\xi)
      \biggr)
      \\
      &=\Lim_{M,N}
      \sum_{\substack{-M\leq m\leq M\\-N\leq n\leq N\\m\neq 0}}
      \frac{\mathfrak{F}(m\omega_1)^re^{2\pi i(mx+n(y+rz))}}{(m\omega_1+n\omega_2)^{k}}
      +
      \Res_{\xi=0}\xi^{-k}\mathfrak{D}_r(\xi)
      +
      \lim_{N\to\infty}
      \sum_{\substack{-N\leq n\leq N\\n\neq 0}}
      \Res_{\xi=n\omega_2}\xi^{-k}\mathfrak{D}_r(\xi)
      \\
      &=\Lim_{M,N}
      \sum_{\substack{-M\leq m\leq M\\-N\leq n\leq N\\m\neq 0}}
      \frac{\mathfrak{F}(m\omega_1)^re^{2\pi i(mx+n(y+rz))}}{(m\omega_1+n\omega_2)^{k}}
      +
      \Res_{\xi=0}\xi^{-k}\mathfrak{K}_r(\xi).
    \end{split}
  \end{equation}
  
  Next we show the holomorphy of $\mathfrak{K}_r$ in the neighborhood of the origin.
  We see that
  \begin{equation}
    \begin{split}
      \mathfrak{K}_r(\xi)
      &=
      \mathfrak{D}_r(\xi)
      -
      \Res_{\eta=0}
      \mathfrak{D}_r(\eta)
      \mathfrak{F}(\xi-\eta;\{y+rz\};\omega_2)
      \\
      &=
      \mathfrak{D}_r(\xi)
      -
      \Res_{\eta=0}
      \mathfrak{D}_r(\eta)
      (\xi-\eta)^{-1}
      -
      \Res_{\eta=0}
      \mathfrak{D}_r(\eta)
      \bigl(\mathfrak{F}(\xi-\eta;\{y+rz\};\omega_2)-(\xi-\eta)^{-1}\bigr).
    \end{split}
  \end{equation}
  Since the last term is holomorphic in the neighborhood of the origin,
  it is sufficient to check the first two terms.
  We have
  \begin{equation}
    \begin{split}
      \mathfrak{D}_r(\xi)
      -
      \Res_{\eta=0}
      \mathfrak{D}_r(\eta)
      (\xi-\eta)^{-1}
      &=
      \sum_{j=0}^r \mathcal{D}_j\xi^{j-r-1}+O(1)
      -
      \xi^{-1}
      \Res_{\eta=0}
      \mathfrak{D}_r(\eta)
      (1-\eta/\xi)^{-1}
      \\
      &=
      \sum_{j=0}^r \mathcal{D}_j\xi^{j-r-1}+O(1)
      -
      \xi^{-1}
      \sum_{j=0}^r
      \mathcal{D}_j\xi^{j-r}
      \\
      &=O(1),
    \end{split}
  \end{equation}
which implies the holomorphy of $\mathfrak{K}$ in the neighborhood of the origin.
Hence \eqref{eq:K_taylor} is valid, and so
$$
\Res_{\xi=0}\xi^{-k}\mathfrak{K}_r(\xi)=\frac{{\cal{K}}_{k,r}}{k!}.
$$
From this and \eqref{4-39} we obtain the conclusion \eqref{K_kq-1}.


  If $k\geq3$, or $k=2$ and $0<z<1$,
  then
  the series converges absolutely uniformly and
  we have the result.

\end{proof}


\section{Relations among $\mathcal{K}_{k,r}$, $\mathcal{H}_l$ and
$\mathcal{B}^{\laa r \raa}_{k}$}\label{sec-3-3}

In this section, we first give the following result 
which includes the previous result given in \cite{TsBul}. At the end of this section, we will give the proof of Theorem \ref{T-1-1}.

\begin{theorem}
\label{Main-theorem-2}
Assume $0\leq z\leq 1$ and let $k\in\mathbb{N}$.
For $(x,y)\neq(0,0)$.
  \begin{multline}
    \label{K_kq-3}
    \mathcal{K}_{k,r}(x,y,z;\omega_1,\omega_2)
    \\
    \begin{aligned}
      &=  
      k!
      \sum_{l=r+1}^{k+r}
      \frac{\mathcal{H}_l(x,y;\omega_1,\omega_2)}{l!}
      \frac{\mathcal{B}^{\laa r \raa}_{k+r-l}(z;\omega_2)}{(k+r-l)!}
      \\
      &\quad+
      k!
      \sum_{l=0}^{r}
      \frac{\mathcal{H}_l(x,y;\omega_1,\omega_2)}{l!}\\
      &\quad \qquad \times \Bigl(
      \frac{\mathcal{B}^{\laa r \raa}_{k+r-l}(z;\omega_2)}{(k+r-l)!}
      -
      \sum_{j=0}^{r-l}
      \frac{\mathcal{B}^{\laa r \raa}_{r-j-l}(z;\omega_2)}{(r-j-l)!}
      \frac{(-1)^j}{j!}
      \frac{(k+j-1)!}{(k-1)!}\frac{\mathcal{B}_{k+j}(\{y+rz\};\omega_2)}{(k+j)!}
      \Bigr).
    \end{aligned}
  \end{multline}
For $(x,y)=(0,0)$ and $z\neq 1$,
\begin{multline}
\label{K_kq-2}
    \mathcal{K}_{k,r}(0,0,z;\omega_1,\omega_2)
    \\
    \begin{aligned}
      &=  
      k!
      \sum_{l=r+1}^{k+r}
      \frac{\mathcal{H}_l(\omega_1,\omega_2)}{l!}
      \frac{\mathcal{B}^{\laa r \raa}_{k+r-l}(z;\omega_2)}{(k+r-l)!}
      \\
      &\qquad+
      k!
      \sum_{\substack{l=0\\l\neq 1}}^{r}
      \frac{\mathcal{H}_l(\omega_1,\omega_2)}{l!}\\
      &\qquad \qquad \times 
      \Bigl(
      \frac{\mathcal{B}^{\laa r \raa}_{k+r-l}(z;\omega_2)}{(k+r-l)!}
      -
      \sum_{j=0}^{r-l}
      \frac{\mathcal{B}^{\laa r \raa}_{r-j-l}(z;\omega_2)}{(r-j-l)!}
      \frac{(-1)^j}{j!}
      \frac{(k+j-1)!}{(k-1)!}\frac{\mathcal{B}_{k+j}(\{rz\};\omega_2)}{(k+j)!}
      \Bigr).
    \end{aligned}
  \end{multline}
\end{theorem}

\begin{proof}
From \eqref{def-H_k-omega} and \eqref{eq:Tay_F} we have
  \begin{equation*}
    \begin{split}
      \mathfrak{E}(\xi) \mathfrak{F}(\xi)^r
      &=
      \Bigl(\sum_{l=0}^\infty \frac{\mathcal{H}_l}{l!}\xi^{l-1}\Bigr)
      \Bigl(\sum_{k=0}^\infty \frac{\mathcal{B}^{\laa r \raa}_k}{k!}\xi^{k-r}\Bigr)
      =
      \sum_{j=0}^\infty
      \Bigl(
      \sum_{l=0}^j
      \frac{\mathcal{H}_l}{l!}
      \frac{\mathcal{B}^{\laa r \raa}_{j-l}}{(j-l)!}
      \Bigr)\xi^{j-r-1},
    \end{split}
  \end{equation*}
and hence
  \begin{equation*}
    \begin{split}
      \mathfrak{K}_r(\xi)
      &=
      \mathfrak{D}_r(\xi)-\Res_{\eta=0}
      \mathfrak{D}_r(\eta)
      \mathfrak{F}(\xi-\eta;\{y+rz\};\omega_2)
      \\
      &= 
      \sum_{k=-r}^\infty
      \Bigl( 
      \sum_{l=0}^{k+r} 
      \frac{\mathcal{H}_l}{l!} 
      \frac{\mathcal{B}^{\laa r \raa}_{k+r-l}}{(k+r-l)!}  
      \Bigr)\xi^{k-1}  \\
      &\qquad
      -
      \sum_{j=0}^r
      \Bigl(
      \sum_{l=0}^{r-j}
      \frac{\mathcal{H}_l}{l!}
      \frac{\mathcal{B}^{\laa r \raa}_{r-j-l}}{(r-j-l)!}
      \Bigr)
      \frac{(-1)^j}{j!}
      \frac{d^j}{d\xi^j}\mathfrak{F}(\xi;\{y+rz\};\omega_2)
      \\
      &=  
      \sum_{k=1}^\infty
      \Bigl(
      \sum_{l=0}^{k+r}
      \frac{\mathcal{H}_l}{l!}
      \frac{\mathcal{B}^{\laa r \raa}_{k+r-l}}{(k+r-l)!}
      \Bigr)\xi^{k-1}
      \\
      &\qquad
      -
      \sum_{j=0}^r
      \Bigl(
      \sum_{l=0}^{r-j}
      \frac{\mathcal{H}_l}{l!}
      \frac{\mathcal{B}^{\laa r \raa}_{r-j-l}}{(r-j-l)!}
      \Bigr)
      \frac{(-1)^j}{j!}
      \sum_{k=1}^\infty \frac{(k+j-1)!}{(k-1)!}\frac{\mathcal{B}_{k+j}(\{y+rz\};\omega_2)}{(k+j)!}\xi^{k-1},
    \end{split}
  \end{equation*}
where on the first sum on the right-most side, the part $-r\leq k\leq 0$ 
is removed because $\mathfrak{K}_r(\xi)$ is holomorphic by Theorem 
\ref{Main-theorem}.
Comparing the coefficients of $\xi^{k-1}$ we obtain \eqref{K_kq-3}.
Using the continuity \eqref{eq:K_F1}, we obtain \eqref{K_kq-2}.
\end{proof}

\begin{theorem}
\label{thm:diff_eq}
  For $z\not\in(\mathbb{Z}-y)/r$,
  \begin{equation}
    \label{eq:diff_eq_gen}
    \frac{\omega_2}{2\pi ir\xi}
    \frac{\partial}{\partial z}
    \mathfrak{K}_r(\xi;x,y,z;\omega_1,\omega_2)
    =
    \mathfrak{K}_r(\xi;x,y,z;\omega_1,\omega_2),
  \end{equation}
and
\begin{equation}
  \label{eq:diff_eq_Ber}
  \frac{\omega_2}{2\pi ir}
  \frac{\partial}{\partial z}
  \mathcal{K}_{k,r}(x,y,z;\omega_1,\omega_2)
= k
  \mathcal{K}_{k-1,r}(x,y,z;\omega_1,\omega_2)\qquad(k\geq 2).
\end{equation}
Furthermore $\mathcal{K}_{k,r}(x,y,z;\omega_1,\omega_2)$ is a polynomial function of degree 
at most $(k-1)$ in $z$ in the interval where $[y+rz]$ is constant.
The degree is $(k-1)$ if and only if $\mathfrak{K}_r(0;x,y,z;\omega_1,\omega_2)\neq0$.
\end{theorem}
\begin{proof}
Fix $x$ and $y$.
By definition, we see that $\mathfrak{K}_r(\xi)$ can be rewritten as the form
\begin{equation}\label{siki5-5}
\mathfrak{K}_r(\xi)=e^{2\pi i\xi rz/\omega_2}\widetilde{K}_r(\xi), 
\end{equation}
where 
$\widetilde{K}_r(\xi)$ is independent of $z$
in the interval where $[y+rz]$ is constant. 
This expression yields \eqref{eq:diff_eq_gen} and hence
\eqref{eq:diff_eq_Ber}.
Furthermore,
$\widetilde{K}_r(\xi)=e^{-2\pi i\xi rz/\omega_2}\mathfrak{K}_r(\xi)$ and 
\eqref{eq:K_taylor}
imply that $\widetilde{K}_r(\xi)$ is holomorphic in $\xi$ around the origin.
Comparing the Taylor expansions of the both sides of \eqref{siki5-5},
we see that the degree of
$\mathcal{K}_{k,r}(x,y,z;\omega_1,\omega_2)$ in $z$ is at most $(k-1)$.
By \eqref{eq:diff_eq_Ber}, we have
\begin{equation}
  \mathcal{K}_{k,r}(x,y,z;\omega_1,\omega_2)=
\Bigl(\frac{2\pi ir}{\omega_2}\Bigr)^{k-1} k
  \mathcal{K}_{1,r}(x,y,z;\omega_1,\omega_2)z^{k-1}+O(z^{k-2}).
\end{equation}
Noting $\mathcal{K}_{1,r}(x,y,z;\omega_1,\omega_2)=\mathfrak{K}_r(0)$, we have 
the last assertion of the theorem.
\end{proof}

Now from Theorems \ref{Main-theorem}, \ref{Main-theorem-2} and \ref{thm:diff_eq} we 
can immediately deduce Theorem \ref{T-1-1} as follows.

\begin{proof}[Proof of Theorem \ref{T-1-1}]
For $r\geq2$,
from \eqref{Gene-Hur} and \eqref{K_kq-1} we have
  \begin{equation}\label{nagasugiru0}
      \frac{\mathcal{K}_{k,r}(0,0,z;1,i)}{k!}
      =
      -\sum_{\substack{(m,n)\in\mathbb{Z}^2\\m\neq 0}}
      \Bigl(
      \pi\frac{(-1)^n}{\sinh(m\pi)}
      \Bigr)^r
      \frac{e^{2\pi (m+in)r(z-1/2)}}{(m+ni)^{k}},
  \end{equation}
while from \eqref{K_kq-2} we have 
\begin{multline}\label{nagasugiru1}
    \frac{\mathcal{K}_{k,r}(0,0,z;1,i)}{k!}
    \\
    \begin{aligned}
      &=  
      \sum_{l=r+1}^{k+r}
      \frac{\mathcal{H}_l(1,i)}{l!}
      \frac{\mathcal{B}^{\laa r \raa}_{k+r-l}(z;i)}{(k+r-l)!}
      \\
      &\qquad+
      \sum_{\substack{l=0\\l\neq 1}}^{r}
      \frac{\mathcal{H}_l(1,i)}{l!}
      \Bigl(
      \frac{\mathcal{B}^{\laa r \raa}_{k+r-l}(z;i)}{(k+r-l)!}
      -
      \sum_{j=0}^{r-l}
      \frac{\mathcal{B}^{\laa r \raa}_{r-j-l}(z;i)}{(r-j-l)!}
      \frac{(-1)^j}{j!}
      \frac{(k+j-1)!}{(k-1)!}\frac{\mathcal{B}_{k+j}(\{rz\};i)}{(k+j)!}
      \Bigr)
    \end{aligned}
\end{multline}
which is, by Lemma \ref{lem.2.3} and \eqref{2-33},
further equal to

\begin{multline}\label{nagasugiru2}
   \begin{aligned}
     &-
      \sum_{l=r+1}^{k+r}
      \frac{(2\varpi)^lH_l}{l!}
      \frac{(2\pi)^{k+r-l}B^{\laa r \raa}_{k+r-l}(z)}{(k+r-l)!}
      \\
      &-
      \sum_{l=4}^{r}
      \frac{(2\varpi)^lH_l}{l!}(2\pi)^{k+r-l}
      \Bigl(
      \frac{B^{\laa r \raa}_{k+r-l}(z)}{(k+r-l)!}
      -
      \sum_{j=0}^{r-l}
      \frac{B^{\laa r \raa}_{r-j-l}(z)}{(r-j-l)!}
      \frac{(-1)^j}{j!}
      \frac{(k+j-1)!}{(k-1)!}\frac{B_{k+j}(\{rz\})}{(k+j)!}
      \Bigr)
      \\
      &+
      \frac{2\pi}{2!}(2\pi)^{k+r-2}
      \Bigl(
      \frac{B^{\laa r \raa}_{k+r-2}(z)}{(k+r-2)!}
      -
      \sum_{j=0}^{r-2}
      \frac{B^{\laa r \raa}_{r-j-2}(z)}{(r-j-2)!}
      \frac{(-1)^j}{j!}
      \frac{(k+j-1)!}{(k-1)!}\frac{B_{k+j}(\{rz\})}{(k+j)!}
      \Bigr)
      \\
      &+
      (2\pi)^{k+r}
      \Bigl(
      \frac{B^{\laa r \raa}_{k+r}(z)}{(k+r)!}
      -
      \sum_{j=0}^{r}
      \frac{B^{\laa r \raa}_{r-j}(z)}{(r-j)!}
      \frac{(-1)^j}{j!}
      \frac{(k+j-1)!}{(k-1)!}\frac{B_{k+j}(\{rz\})}{(k+j)!}
      \Bigr).
    \end{aligned}
  \end{multline}
From \eqref{nagasugiru0} and \eqref{nagasugiru2} we obtain
\eqref{eq:formula_0}.
For $r=1$, a similar calculation yields
\eqref{eq:formula_1}.
 Note that in the case when $k=1$, we interpret 
$\sum_{\substack{(m,n)\in\mathbb{Z}^2\\m\neq 0}}$
as
$\Lim_{M,N}
\sum_{\substack{-M\leq m\leq M\\-N\leq n\leq N\\m\neq 0}}$,
which is convergent. 
The statement for the degree
  with respect to $\pi$, $\varpi^4$ 
  is clear.
  As for $z$, it follows from Theorem \ref{thm:diff_eq}.
\end{proof}

\bigskip

\section{Explicit examples} \label{sec-4}

From Theorem \ref{T-1-1}, we can give the following explicit formulas.

\begin{example}
For $z\in \mathbb{R}$ with $0\leq z\leq 1$, from \eqref{eq:formula_1} we have 
  \begin{align}
    \label{ex-eq-1}
    &\pi\sum_{\substack{(m,n)\in\mathbb{Z}^2\\m\neq 0}}
    \frac{(-1)^n}{\sinh(m\pi)}
    \frac{e^{2\pi (m+in)(z-1/2)}}{(m+ni)^{2}}
    =
    \left(\frac{2 \pi ^3}{3}-2 \pi ^2\right) z
    -\frac{\pi ^3}{3}+\pi ^2 \qquad(z\neq0,1),
    \\[3truemm]
\label{ex-eq-2}
    &\pi\sum_{\substack{(m,n)\in\mathbb{Z}^2\\m\neq 0}}
    \frac{(-1)^n}{\sinh(m\pi)}
    \frac{e^{2\pi (m+in)(z-1/2)}}{(m+ni)^{3}}  
    \\
\notag 
    &\qquad\qquad=
    \left(\frac{2 \pi ^4}{3}-2 \pi ^3\right) z^2
    +\left(-\frac{2 \pi ^4}{3}+2 \pi^3\right) z 
    +\frac{\varpi^4}{15}
    +\frac{4 \pi ^4}{45}-\frac{\pi ^3}{3},
    \\[3truemm]
\label{ex-eq-3}
    &\pi\sum_{\substack{(m,n)\in\mathbb{Z}^2\\m\neq 0}}
    \frac{(-1)^n}{\sinh(m\pi)}
    \frac{e^{2\pi (m+in)(z-1/2)}}{(m+ni)^{4}}
\\
\notag
&\qquad\qquad
    =
    \left(\frac{4 \pi ^5}{9}-\frac{4 \pi ^4}{3}\right)z^3
    +\left(-\frac{2 \pi ^5}{3}+2 \pi ^4\right) z^2
\\
\notag
&\qquad\qquad\qquad
    +\left(\frac{2 \pi \varpi^4}{15}+\frac{8 \pi ^5}{45}-\frac{2 \pi ^4}{3}\right) z
    -\frac{\pi  \varpi^4}{15}+\frac{\pi^5}{45}.
  \end{align}
The right-hand sides of the above formulas are indeed polynomials in $z$.

In particular, putting $z=1/2$ in \eqref{ex-eq-2}, we obtain \eqref{1-11}. 
Considering the cases $z=0$ and $z=1$ in \eqref{ex-eq-2} and adding them, we obtain 
\begin{align}
& \sum_{\substack{(m,n)\in\mathbb{Z}^2\\m\neq 0}}
    \frac{\coth(m\pi)}{(m+ni)^{3}}=\frac{\varpi^4}{15\pi}+\frac{4\pi^3}{45}-\frac{\pi^2}{3}, \label{aust-1}
\end{align}
which was obtained in \cite{TsAust}. 
\end{example}

\begin{example}\label{Exam-fin}
For simplicity, we let 
$\mathcal{G}_k^{\laa r \raa}(\tau)=\mathcal{G}_k^{\laa r \raa}(0,0,1/2;1,\tau)$. 
From \eqref{eq:formula_0} we can evaluate $\mathcal{G}_k^{\laa r \raa}(i)$ 
for $r \geq 2$.   To evaluate the right-hand side of \eqref{eq:formula_0} in the
following cases, we need to know some special values of $B_k^{\laa r \raa}(1/2)$:
\begin{align*}
&B_0(1/2)=1,\; B_2(1/2)=-1/12,\; B_4(1/2)=7/240,\; B_6(1/2)=-31/1344,\\
&B_0^{\laa 2 \raa}(1/2)=1,\; B_2^{\laa 2 \raa}(1/2)=-1/6,\; 
B_4^{\laa 2 \raa}(1/2)=1/10,\; B_6^{\laa 2 \raa}(1/2)=-5/42,\\
&B_0^{\laa 3 \raa}(1/2)=1,\; B_2^{\laa 3 \raa}(1/2)=-1/4,\;    
B_4^{\laa 3 \raa}(1/2)=17/80,\; B_6^{\laa 3 \raa}(1/2)=-457/1344,\\
&B_0^{\laa 4 \raa}(1/2)=1,\; B_2^{\laa 4 \raa}(1/2)=-1/3,\;    
B_4^{\laa 4 \raa}(1/2)=11/30,\; B_6^{\laa 4 \raa}(1/2)=-31/42,\\
\end{align*}
while $B_k^{\laa r \raa}(1/2)=0$ for any odd $k$.  These values can be obtained
directly from the definition \eqref{Ber-high}. 
When $r\not\equiv k$ (mod 2), we can easily confirm that 
$\mathcal{G}_k^{\laa r \raa}(i)$ vanishes, hence we only give examples in which
$r\equiv k$ (mod 2). 

\begin{align}
& \mathcal{G}_2^{\laa 2 \raa}(i)=\sum_{\substack{(m,n)\in\mathbb{Z}^2\\m\neq 0}}
    \frac{1}{\sinh(m\pi)^2(m+ni)^{2}}
    =\frac{\varpi^4}{15\pi^2}-\frac{11}{45}\pi^2+\frac{2}{3}\pi, \label{4-2}\\
& \mathcal{G}_4^{\laa 2 \raa}(i)=\sum_{\substack{(m,n)\in\mathbb{Z}^2\\m\neq 0}}
    \frac{1}{\sinh(m\pi)^2(m+ni)^{4}}=-\frac{\varpi^4}{45}+\frac{37}{945}\pi^4-\frac{4}{45}\pi^3, \label{4-3}\\
& \mathcal{G}_3^{\laa 3 \raa}(i)=\sum_{\substack{(m,n)\in\mathbb{Z}^2\\m\neq 0}}
    \frac{(-1)^n}{\sinh(m\pi)^3(m+ni)^{3}}=-\frac{\varpi^4}{30\pi}+\frac{151}{1890}\pi^3-\frac{1}{5}\pi^2, \label{4-5}\\
& \mathcal{G}_2^{\laa 4 \raa}(i)=\sum_{\substack{(m,n)\in\mathbb{Z}^2\\m\neq 0}}
    \frac{1}{\sinh(m\pi)^4(m+ni)^{2}}=-\frac{\varpi^4}{15\pi^2}+\frac{191}{945}\pi^2-\frac{8}{15}\pi, \label{4-6}
\end{align}
and 
\begin{align}
& \mathcal{G}_1^{\laa 3 \raa}(i)=\Lim_{M,N}\sum_{\substack{-M\leq m \leq M \\
-N\leq n \leq N\\m\neq 0}}
    \frac{(-1)^n}{\sinh(m\pi)^3(m+ni)}=\frac{\varpi^4}{15\pi^3}-\frac{11}{45}\pi+\frac{2}{3}, \label{4-4}\\
& \mathcal{G}_{1}^{\laa 5 \raa}(i)=\Lim_{M,N}\sum_{\substack{-M\leq m \leq M \\
-N\leq n \leq N\\m\neq 0}}
    \frac{(-1)^n}{\sinh(m\pi)^5(m+ni)}=-\frac{\varpi^4}{15\pi^3}+\frac{191}{945}\pi-\frac{8}{15},\label{4-4-2}
\end{align}
which are higher-order versions of our previous results in \cite{TsBul}, where
we have proved, for example, \eqref{1-11},\ \eqref{1-11-2} and
\begin{align}
& \mathcal{G}_{1}^{\laa 1 \raa}(i)=\Lim_{M,N}\sum_{\substack{-M\leq m \leq M \\
-N\leq n \leq N\\m\neq 0}}
    \frac{(-1)^n}{\sinh(m\pi)(m+ni)}=\frac{\pi}{3}-1.\label{4-4-3}
\end{align}
Note that, by using the same method as introduced in \cite{TsBul}, we can prove that 
\begin{equation}
\mathcal{G}_1^{\laa 2k+1 \raa}(i) = \frac{1}{\pi} \mathcal{G}_2^{\laa 2k \raa}(i) \qquad (k\in \mathbb{N}). \label{rel-G1}
\end{equation}
In fact, from the above listed equations we can observe 
that \eqref{rel-G1} for $k=1, 2$ is true. 
\end{example}

Next we consider more general cases.

\begin{example}
\label{Exam-rho} 
Putting $(k,r,x,y,z,\omega_1,\omega_2)=(1,1,0,0,1/2,1,\tau)$ in \eqref{K_kq-2}, 
and using \eqref{eq:H0} and \eqref{2-33}, we have
\begin{equation}
\mathcal{K}_{1,1}(0,0,1/2;1,\tau)=\frac{1}{2}\mathcal{H}_2(1,\tau)+\frac{\pi^2}{3\tau^2}. \label{e-11-2}
\end{equation}
Therefore, by using \eqref{H_2-rel}, we obtain
\begin{equation}
\mathcal{K}_{1,1}(0,0,1/2;1,-1/\tau)-\tau^2 \mathcal{K}_{1,1}(0,0,1/2;1,\tau)=-2\pi i\tau+\frac{\pi^2}{3}\left(\tau^2-1\right).\label{e-11}
\end{equation}
On the other hand, letting $(k,r)=(1,1)$ in \eqref{K_kq-1}, we have
\begin{equation}
\label{e-12}
\mathcal{K}_{1,1}(0,0,1/2;1,\tau)=-\frac{\pi i}{\tau}\mathcal{G}_1^{\laa 1 \raa}(\tau).
\end{equation}
Substituting \eqref{e-12} in the both cases $\tau$ and $-1/\tau$ into \eqref{e-11}, 
we have the following reciprocity formula: 
\begin{equation}
\mathcal{G}_1^{\laa 1 \raa}(\tau)+\mathcal{G}_1^{\laa 1 \raa}(-1/\tau)=-2+\frac{\tau^2-1}{3\tau i}\pi. \label{Henkan-f}
\end{equation}
Letting $\tau=i$, we immediately obtain \eqref{4-4-3}. 
Next we let $\tau=\rho=e^{2\pi i/3}$. We prove
\begin{equation}
\label{e-13}
\mathcal{G}_1^{\laa 1 \raa}(-1/\rho)=-\frac{1}{\rho} \mathcal{G}_1^{\laa 1 \raa}(\rho).
\end{equation}
In fact, 
\begin{equation}
\label{e-13a}
\begin{split}
\mathcal{G}_1^{\laa 1 \raa}(-1/\rho)& = 
\Lim_{M,N}\sum_{\substack{-M\leq m \leq M \\
-N\leq n \leq N\\m\neq 0}}
\frac{(-1)^{n}}{\sinh(m\pi i(-\rho)) (m-n\rho^{-1})}\\
& =\Lim_{M,N}\sum_{\substack{-M\leq m \leq M \\
-N\leq n \leq N\\m\neq 0}}
\frac{(-1)^{n}}{\sinh(m\pi i(1/\rho+1)) (m(-\rho-\rho^2)-n\rho^{2})},
\end{split}
\end{equation}
because $\rho^{-1}=\rho^2=-\rho-1$.   
Since any of the three meanings of $\Lim_{M,N}$ gives the same value, here we
specify that $\Lim_{M,N}$ on the right-hand side of \eqref{e-13a} means the limit
$\lim_{M\to\infty}\lim_{N\to\infty}$.   Then
the right-hand side of \eqref{e-13a} is further equal to
\begin{equation}\label{e-13b}
\begin{split}
& \lim_{M\to\infty}\sum_{\substack{-M\leq m \leq M \\m\neq 0}}\lim_{N\to\infty}
\sum_{-N\leq n \leq N}
\frac{(-1)^{m+n}}{\sinh(m\pi i/\rho) (-m\rho-(m+n)\rho^2)}\\
&=-\frac{1}{\rho}\lim_{M\to\infty}\sum_{\substack{-M\leq m \leq M \\m\neq 0}}
\lim_{N\to\infty}\sum_{-N+m\leq n \leq N+m} 
A(m,n;\rho),
\end{split}
\end{equation}
where
$$
A(m,n;\rho)=\frac{(-1)^{n}}{\sinh(m\pi i/\rho) (m+n\rho)}.
$$
Fix an $M$ temporarily, and replace the condition $-N+m\leq n \leq N+m$ on the
inner sum by $-N\leq n \leq N$.   When $m>0$, the terms $A(m,n;\rho)$ for 
$N<n\leq N+m$ 
are to be removed, while the terms for
$-N<n\leq -N+m$ 
are to be added.   The total number of these terms is $O(M)$, and each of these
terms is estimated as $O(N^{-1})$ if $N$ is sufficiently large.   Hence the total
contribution of these terms tends to 0 when $M$ is fixed and $N\to\infty$.
The case $m<0$ is similar.
Therefore we can replace the condition on the inner sum by  $-N\leq n \leq N$
on the right-hand side of \eqref{e-13b}.   Then it is equal to
$-\rho^{-1}\mathcal{G}_1^{\laa 1 \raa}(\rho)$.   This proves \eqref{e-13}.

Putting $\tau=\rho$ in \eqref{Henkan-f} and combining with \eqref{e-13}, we have
$$\left( 1-\frac{1}{\rho}\right)\mathcal{G}_1^{\laa 1 \raa}(\rho)=-2+\frac{\rho^2-1}{3\rho i}\pi.$$
Therefore we have 
\begin{equation}
\label{e-16}
\mathcal{G}_1^{\laa 1 \raa}(\rho)=\frac{i}{\rho}\left(\frac{\pi}{3}+\frac{2i\rho^2}
{\rho-1}\right)=\frac{i(\pi-2\sqrt{3})}{3\rho}.
\end{equation}
From the Legendre formula for quasi-periods of the Weierstrass 
zeta-function (see, for example, \cite{HC}), we can deduce $G_2(\rho)=2\pi\rho/\sqrt{3}$.
It should be noted that \eqref{e-16} can also be obtained from this value of 
$G_2(\rho)$ and \eqref{e-11-2}.  A feature of the above proof of \eqref{e-16} is that it does not use the value of $G_2(\rho)$.

Next we add a few more examples.
Similarly to \eqref{1-4}, it is known that
\begin{equation}
\begin{split}
& G_6(\rho)=\frac{\widetilde{\varpi}^6}{35},\ \  G_{12}(\rho)=\frac{\widetilde{\varpi}^{12}}{7007},\ \ G_{18}(\rho)=\frac{\widetilde{\varpi}^{18}}{1440257},\cdots
\end{split}
\label{G-rho}
\end{equation}
and $G_{2j}(\rho)=0$ for $j\geq 2$ with $3\nmid j$, where
$$\widetilde{\varpi}= 2 \int_{0}^{1} \frac{1}{\sqrt{1-x^6}} dx=\frac{\Gamma(1/3)^3}{2^{4/3}\pi}=2.428650648\cdots.$$
These facts were first studied by Matter (see Lemmermeyer \cite[Chapter 8]{Lem}; also 
Katayama \cite[\S 6.1]{Katayama}). Hence, from \eqref{H_k}, we have
\begin{align*}
& \mathcal{H}_{6}(1,\rho)=-6!G_6(\rho)=-\frac{144}{7}\widetilde{\varpi}^6,\quad 
\mathcal{H}_{12}(1,\rho)=-12!G_{12}(\rho)=-\frac{6220800}{91}\widetilde{\varpi}^{12},
\ldots
\end{align*}
and $\mathcal{H}_{2j}(1,\rho)=0$ for $j\geq 2$ with $3\nmid j$. Therefore, 
applying the same argument as in the case of $\mathcal{G}_1^{\laa 1 \raa}(\rho)$, 
we obtain, for example,
\begin{align}
\mathcal{G}_3^{\laa 1 \raa}(\rho)&=
\sum_{\substack{(m,n)\in\mathbb{Z}^2\\m\neq 0}}
\frac{(-1)^n}{\sinh(m\pi i/\rho) (m+n\rho)^3}=i\left(\frac{7}{90}\pi^3-\frac{\sqrt{3}}{9}\pi^2\right), \label{e-17} \\
\mathcal{G}_5^{\laa 1 \raa}(\rho)&=
\sum_{\substack{(m,n)\in\mathbb{Z}^2\\m\neq 0}}
\frac{(-1)^n}{\sinh(m\pi i/\rho) (m+n\rho)^5} =\rho i\left(-\frac{\widetilde{\varpi}^6}{35\pi}+\frac{31}{2520}\pi^5-\frac{7\sqrt{3}}{540}\pi^4\right), \label{e-18} \\
\mathcal{G}_2^{\laa 2 \raa}(\rho)&=
\sum_{\substack{(m,n)\in\mathbb{Z}^2\\m\neq 0}}
\frac{1}{(\sinh(m\pi i/\rho))^2 (m+n\rho)^2} =\rho \left(\frac{11}{45}\pi^2-\frac{4\sqrt{3}}{9}\pi\right), \label{e-19} \\
\mathcal{G}_4^{\laa 2 \raa}(\rho)&=
\sum_{\substack{(m,n)\in\mathbb{Z}^2\\m\neq 0}}
\frac{1}{(\sinh(m\pi i/\rho))^2 (m+n\rho)^4} =\frac{1}{\rho}\left(-\frac{\widetilde{\varpi}^6}{35\pi^2}+\frac{37}{945}\pi^4-\frac{8\sqrt{3}}{135}\pi^3\right). \label{e-20} 
\end{align}
Note that \eqref{e-16} and \eqref{e-17}-\eqref{e-20} can also be given by the same method as 
in \cite{TsBul}. 
\end{example}

\begin{example}\label{Exam-Aust}
By combining \eqref{1-1} and \eqref{4-3} and using $(\cosh x)^2=(\sinh x)^2+1$, we obtain
\begin{equation}
\begin{split}
\sum_{\substack{(m,n)\in\mathbb{Z}^2\\m\neq 0}}\frac{(\coth(m\pi))^2}{(m+ni)^4}
&= \sum_{\substack{(m,n)\in\mathbb{Z}^2\\m\neq 0}}
    \frac{1+(\sinh(m\pi))^{-2}}{(m+ni)^4}\\
&=\Biggl(\sum_{\substack{(m,n)\in\mathbb{Z}^2\\(m,n)\neq (0,0)}}-
\sum_{\substack{(m,n)\in\mathbb{Z}^2\\m=0,n\neq 0}}\Biggr)\frac{1}{(m+ni)^4}
+\sum_{\substack{(m,n)\in\mathbb{Z}^2\\m\neq 0}}\frac{(\sinh(m\pi))^{-2}}
{(m+ni)^4}\\
&=\frac{2}{45}\varpi^4+\frac{16}{945}\pi^4-\frac{4}{45}\pi^3,
\end{split}
\label{aust-2}
\end{equation}
which was obtained in \cite{TsAust}. Similarly, by \eqref{e-20}, we obtain
\begin{equation}
\begin{split}
\sum_{\substack{(m,n)\in\mathbb{Z}^2\\m\neq 0}}\frac{(\coth(m\pi i/\rho))^2}
{(m+n\rho)^4}&= \sum_{\substack{(m,n)\in\mathbb{Z}^2\\m\neq 0}}
\frac{1+(\sinh(m\pi i/\rho))^{-2}}{(m+n\rho)^4}\\
&=\frac{1}{\rho}\left\{ -\frac{\widetilde{\varpi}^6}{35\pi^2}+\frac{16}{945}\pi^4-\frac{8\sqrt{3}}{135}\pi^3\right\},
\end{split}
\label{aust-3}
\end{equation}
because $G_4(\rho)=0$.
\end{example}

\bigskip

\section{Values of $q$-zeta functions at positive integers} \label{sec-5}

In \cite{KKW}, Kaneko, Kurokawa and Wakayama defined a $q$-analogue of the Riemann zeta function, so-called the $q$-zeta function, by 
\begin{equation}
\zeta_q(s)={(1-q)^s}\sum_{m=1}^\infty \frac{q^{m(s-1)}}{(1-q^m)^s}  \label{5-1}
\end{equation}
for $q\in \mathbb{R}$ with $q\leq 1$. They showed that $\lim_{q\to 1}\zeta_q(s)=\zeta(s)$ 
for all $s \in \mathbb{C}$ except for $s=1$. More generally, they studied the function
\begin{equation}
f_q(s,t)=(1-q)^s\sum_{m=1}^\infty \frac{q^{mt}}{(1-q^m)^s}. \label{5-2}
\end{equation}
Note that $f_q(s,s-1)=\zeta_q(s)$. 
In \cite{WY}, Wakayama and Yamasaki showed that $\lim_{q\to 1}f_q(s,t)=\zeta(s)$ for all $(s,t)\in \mathbb{C}^2$ except for $s=1$. Therefore $f_q(s,t)$ can be regarded as a true $q$-analogue of $\zeta(s)$. 

In this section, we aim to evaluate $f_q(2k,k)$ for $k\in \mathbb{N}$ when $q=e^{-2\pi}$. From the definition, we can see that if $q=e^{-2\pi}$ then 
\begin{equation}
f_q(2k,k)=\left(\frac{ 1-e^{-2\pi}}{2}\right)^{2k}\sum_{m=1}^\infty \frac{1}{\sinh(m\pi)^{2k}}\quad (k\in \mathbb{N}). \label{5-3}
\end{equation}
Therefore it is necessary to evaluate $\sum_{m\geq 1} \sinh(m\pi)^{-2k}$ for 
$k\in \mathbb{N}$. 

\begin{prop} \label{P-5-1}
For $k\in \mathbb{N}$,
\begin{equation}
\sum_{m\in \mathbb{Z}\setminus \{0\}} \frac{1}{\sinh(m\pi)^{2k}}=\frac{1}{\pi} \mathcal{G}_{1}^{\laa 2k-1 \raa}(i). \label{5-4}
\end{equation}
\end{prop}

\begin{proof}
Let 
$$\mathfrak{S}(j)=\sum_{n=1}^\infty \frac{(-1)^n}{\sinh(n\pi)n^{j}}\quad (j \in \mathbb{Z}).$$
Cauchy \cite{Ca} showed that
\begin{equation}
\mathfrak{S}(-1)=-\frac{1}{4\pi},\ \mathfrak{S}(-4j-1)=0\ \ (j\in \mathbb{N}). \label{5-5}
\end{equation}
For $\theta \in (-\pi,\pi)\subset \mathbb{R}$, let 
\begin{equation*}
\mathcal{J}(\theta)=\sum_{n=1}^\infty \frac{(-1)^n\left\{ \sin(n\theta)+\sinh(n\theta)\right\}}{\sinh(n\pi)}+\frac{\theta}{2\pi},
\end{equation*}
which is absolutely convergent for $\theta \in (-\pi,\pi)$. 
We can easily check that $\sin x+\sinh x=2\sum_{l\geq 0}x^{4l+1}/(4l+1)!$. Hence, by \eqref{5-5}, we have
\begin{equation}
\begin{split}
\mathcal{J}(\theta) =& 2\sum_{l=0}^\infty \mathfrak{S}(-4l-1)\frac{\theta^{4l+1}}{(4l+1)!}+\frac{\theta}{2\pi} \\
=& 2\mathfrak{S}(-1)\theta +\frac{\theta}{2\pi}=0\quad (\theta \in (-\pi,\pi)).
\end{split}
\label{5-6}
\end{equation}
For $k\in \mathbb{N}$ and $\theta \in (-\pi,\pi)\subset \mathbb{R}$, let 
\begin{equation}
\mathcal{F}(k;\theta)=2i \sum_{m\in \mathbb{Z}\setminus \{0\}}\frac{(-1)^me^{im\theta}}{\sinh(m\pi)^{2k-1}}\cdot \mathcal{J}(\theta). \label{5-7}
\end{equation}
From \eqref{5-6}, we have $\mathcal{F}(k;\theta)=0$ for $\theta \in (-\pi,\pi)$. 
On the other hand, combining \eqref{5-5} and \eqref{5-7}, we have
\begin{equation}
\begin{split}
\mathcal{F}(k;\theta)& =\sum_{m\not=0}\sum_{n\geq 1} 
\frac{(-1)^{m+n}\left\{ e^{i(m+n)\theta}-e^{i(m-n)\theta}+i(e^{(im+n)\theta}-
e^{(im-n)\theta})\right\}}{\sinh(m\pi)^{2k-1}\sinh(n\pi)} \\
& \qquad +\frac{i\theta}{\pi}\sum_{m\not=0}\frac{(-1)^m e^{im\theta}}{\sinh(m\pi)^{2k-1}}\\
& =\sum_{m\not=0}\sum_{n\not=0} \frac{(-1)^{m+n}\left\{ e^{i(m+n)\theta}+
ie^{(im+n)\theta}\right\}}{\sinh(m\pi)^{2k-1}\sinh(n\pi)} +\frac{i\theta}{\pi}
\sum_{m\not=0}\frac{(-1)^m e^{im\theta}}{\sinh(m\pi)^{2k-1}}.
\end{split}
\label{5-8}
\end{equation}
In order to integrate this function, we remove the constant term of the first sum, namely the term of $m+n=0$, on the right-hand side of \eqref{5-8} and define
\begin{equation}
\begin{split}
\widetilde{\mathcal{F}}(k;\theta) & =\sum_{m\not=0}\sum_{n\not=0 \atop m+n\not=0} \frac{(-1)^{m+n}e^{i(m+n)\theta}}{\sinh(m\pi)^{2k-1}\sinh(n\pi)} \\
& +i\sum_{m\not=0}\sum_{n\not=0} \frac{(-1)^{m+n} e^{(im+n)\theta}}{\sinh(m\pi)^{2k-1}\sinh(n\pi)} \\
& +\frac{i\theta}{\pi}\sum_{m\not=0}\frac{(-1)^m e^{im\theta}}{\sinh(m\pi)^{2k-1}}.
\end{split}
\label{5-9}
\end{equation}
Since $\mathcal{F}(k;\theta)=0$ for $\theta \in (-\pi,\pi)$, we have
\begin{equation}
\widetilde{\mathcal{F}}(k;\theta)=\sum_{m\not=0}\frac{1}{\sinh(m\pi)^{2k}} \label{5-10}
\end{equation}
for $\theta \in (-\pi,\pi)$. By calculating $(\widetilde{\mathcal{F}}(\theta)+\widetilde{\mathcal{F}}(-\theta))/2$, we obtain from \eqref{5-9} and \eqref{5-10} that 
\begin{equation}
\begin{split}
& \sum_{m\not=0}\sum_{n\not=0 \atop m+n\not=0} \frac{(-1)^{m+n}\cos((m+n)\theta)}
{\sinh(m\pi)^{2k-1}\sinh(n\pi)} + i\sum_{m\not=0}\sum_{n\not=0} \frac{(-1)^{m+n}
\cosh((im+n)\theta)}{\sinh(m\pi)^{2k-1}\sinh(n\pi)} \\
& \qquad -\frac{\theta}{\pi}\sum_{m\not=0}\frac{(-1)^m \sin(m\theta)}{\sinh(m\pi)^{2k-1}}=\sum_{m\not=0}\frac{1}{\sinh(m\pi)^{2k}}. 
\end{split}
\label{5-11}
\end{equation}
Now we integrate the both sides of \eqref{5-11} with respect to $\theta$ from $-\pi$ to $\pi$. Then, noting
\begin{align*}
& \int_{-\pi}^{\pi}\cos((m+n)\theta)d\theta=0,\ \int_{-\pi}^{\pi}\cosh ((im+n)\theta)d\theta=\frac{2(-1)^m \sinh(n\pi)}{im+n},\\
& \int_{-\pi}^{\pi} \theta\sin (m\theta)d\theta=-\frac{2\pi (-1)^m}{m},
\end{align*}
we obtain 
\begin{equation*}
\begin{split}
& 2i\Lim_{M,N}\sum_{\substack{-M\leq m \leq M \\ -N\leq n \leq N\\m\neq 0}} 
\frac{(-1)^{n}}{\sinh(m\pi)^{2k-1}(im+n)} +2\sum_{m\not=0}\frac{1}{\sinh(m\pi)^{2k-1}m}=2\pi \sum_{m\not=0}\frac{1}{\sinh(m\pi)^{2k}}, 
\end{split}
\end{equation*}
Note that the second sum on the left-hand side corresponds to the case $n=0$ of the first sum on the left-hand side. Therefore, replacing $(m,n)$ by $(-m,-n)$ on the left-hand side, we have
\begin{equation*}
\begin{split}
& 2\mathcal{G}_{1}^{\laa 2k-1 \raa}(i)=2\Lim_{M,N}\sum_{\substack{-M\leq m 
\leq M \\ -N\leq n \leq N\\m\neq 0}} \frac{(-1)^{n}}{\sinh(m\pi)^{2k-1}(m+ni)}  
=2\pi \sum_{m\not=0}\frac{1}{\sinh(m\pi)^{2k}}, 
\end{split}
\end{equation*}
which gives \eqref{5-4}. This completes the proof.
\end{proof} 

\begin{example} \label{Exam-5-2}
As we mentioned in the previous section (see Example \ref{Exam-fin}), we can 
evaluate $\mathcal{G}_{1}^{\laa 2k-1 \raa}(i)$ (see \eqref{4-4}-\eqref{4-4-3}). 
Hence, setting $q=e^{-2\pi}$, and combining \eqref{5-3} and \eqref{5-4}, we can obtain the following evaluation formulas for $q$-zeta functions:
\begin{align}
& f_q(2,1)(=\zeta_q(2))=(1-q)^2\sum_{m=1}^\infty \frac{q^{m}}{(1-q^m)^2}=\frac{ \left(1-e^{-2\pi}\right)^{2}}{8}\left(\frac{1}{3}-\frac{1}{\pi}\right),\label{5-12}\\
& f_q(4,2)=(1-q)^4\sum_{m=1}^\infty \frac{q^{2m}}{(1-q^m)^4}=\frac{ \left(1-e^{-2\pi}\right)^{4}}{32}\left(\frac{\varpi^4}{15\pi^4}-\frac{11}{45}+\frac{2}{3\pi}\right),\label{5-13}\\
& f_q(6,3)=(1-q)^6\sum_{m=1}^\infty \frac{q^{3m}}{(1-q^m)^6}=-\frac{ \left(1-e^{-2\pi}\right)^{6}}{128}\left(\frac{\varpi^4}{15\pi^4}-\frac{191}{945}+\frac{8}{15\pi}\right). \label{5-14}
\end{align}
\end{example}

\ 

We will give more general results about $q$-zeta functions in our forthcoming paper. Indeed we will be able to give general forms of evaluation formulas for $q$-zeta functions in terms of theta-functions.

\bigskip 

\bibliographystyle{amsplain}

\end{document}